\documentclass[10pt]{article}

\usepackage{amsfonts, epsfig, amsmath, amssymb, color,amsthm, mathabx}
\usepackage{mathrsfs}
\usepackage[english]{babel}

\usepackage{mathrsfs}
\usepackage{pgfplots}
\pgfplotsset{width=7cm,compat=newest}
\usepackage{pgfplotstable}
\usetikzlibrary{calc}

\usepackage[T1]{fontenc}
\usepackage{lmodern}

\usepackage[displaymath, mathlines, pagewise]{lineno}
\newcommand{\E}{\mathbf{E}}
\renewcommand{\P}{\mathbf{P}}

\usepackage{color}

\theoremstyle{plain}
\newtheorem{thm}{Theorem}[section]
\newtheorem{lem}[thm]{Lemma}
\newtheorem{prp}[thm]{Proposition}

\theoremstyle{definition}
\newtheorem{rem}[thm]{Remark}
\newtheorem{exa}[thm]{Example}

\newcommand{\e}{\varepsilon}
\newcommand{\be}{\begin{equation}}
\newcommand{\ee}{\end{equation}}
\newcommand{\ben}{\begin{equation*}}
\newcommand{\een}{\end{equation*}}

\newcommand{\ba}{\begin{equation}\begin{aligned}}
\newcommand{\ea}{\end{aligned}\end{equation}}

\DeclareMathOperator{\sgn}{sgn}

\newcommand{\ex}{\mathrm{e}}
\newcommand{\di}{\mathrm{d}}

\newcommand{\cL}{\mathcal{L}}

\newcommand{\rF}{\mathscr{F}}

\newcommand{\bI}{\mathbb{I}}

\newcommand{\bR}{\mathbb{R}}

\allowdisplaybreaks[4]

\let\oldmarginpar\marginpar
\renewcommand{\marginpar}[1]{\oldmarginpar{\scriptsize\texttt{\color{red}{#1}}}}

\numberwithin{equation}{section}

\begin{document}
%\layout
\title{Homogenization of a multivariate diffusion\\ with semipermeable   interfaces}

\date{\today}%\footnote{{\jobname}.tex} }

% \date{\today\footnote{{\jobname}.tex\hfill \textbf{Preliminary version!!! Do not distribute!!!}}}
\author{Olga Aryasova\thanks{Institute of Geophysics, National Academy of Sciences of Ukraine, Palladin ave.\ 32, 03680 Kyiv-142,
Ukraine \emph{and} Institute of Mathematics, Friedrich Schiller University Jena, Ernst--Abbe--Platz 2, 
07743 Jena, Germany \emph{and} Igor Sikorsky Kyiv Polytechnic Institute, Beresteiskyi ave.\ 37, 03056, Kyiv, Ukraine; \texttt{oaryasova@gmail.com}}, \ 
Ilya Pavlyukevich$^{*,}$\thanks{Institute of Mathematics, Friedrich Schiller University Jena, Ernst--Abbe--Platz 2, 
07743 Jena, Germany; \texttt{ilya.pavlyukevich@uni-jena.de}}\ ,
and 
Andrey Pilipenko\thanks{Institute of Mathematics, National Academy of Sciences of Ukraine, Tereshchenkivska Str.\ 3, 01601, Kyiv, Ukraine \emph{and} 
Igor Sikorsky Kyiv Polytechnic Institute, Beresteiskyi ave.\ 37, 03056, Kyiv, Ukraine; \texttt{pilipenko.ay@gmail.com}}}

\footnotetext{\hspace{-.25cm}$^*$Corresponding author}

\maketitle

\begin{abstract}
We study the homogenization problem for a system of stochastic differential equation with local time terms  
that models a multivariate diffusion in presence of semipermeable hyperplane interfaces with oblique penetration. 
We show that this system has a unique weak solution and determine its weak limit as the distances between the interfaces converge to zero.
In the limit, the singular local times terms vanish and give rise to an additional regular \emph{interface-induced} drift.
\end{abstract}

\noindent
\textbf{Keywords:} local times, homogenization, existence and uniqueness of weak solutions,
semipermeable  interfaces, oblique reflection, weak convergence, martingale problem.

\smallskip

\noindent
\textbf{2010 Mathematics Subject Classification:} 60H10, 60F05, 60H17, 60J55 %, 78M40, 80M40

% \tableofcontents

\section{Setting and the main results}

The mathematical homogenization problem usually deals with the study of effective parameters of a system with rapidly varying spatial characteristics.
Its original motivation comes from the analysis of composite materials with periodic structure.
If the period of the structure is very small in comparison to the object's macroscopic size one can consider the material as a new homogeneous substance. 
A typical example here is the analysis of material's thermal conductivity. In mathematical terms, one  solves 
a boundary value problem $-\nabla (A(x/\e)\nabla u^\e(x))=f$ in some domain $G\subseteq \bR^n$ 
with certain boundary conditions on $\partial G$. The matrix (thermal conductivity tensor) $A(\cdot)$ is assumed to be periodic in $\bR^n$, and the 
parameter $\e>0$ is small. The solution $u^\e$ gives the temperature distribution in the domain $G$. 
The goal of homogenization consists in determining the \emph{effective}
thermal conductivity $A$ such that $u^\e\to u$ as $\e\to 0$ where $u$ is the solution of the homogenized equation $A \Delta u=f$. 
There is vast literature devoted to this subject see, e.g., Bensoussan et al.\ \cite{bensoussan1978asymptotic}, Jikov et al.\ \cite{JikKozOle94}, and Berdichevsky et al.\ \cite{BerJikPap99}, Chechkin et al.\ \cite{chechkin2007homogenization} that comprises both analytic and stochastic methods. 

In the present paper, we consider a different type of a stochastic homogenization problem: 
homogenization of a diffusion in the presence of narrowly 
located semipermeable hyperplane interfaces. 
In simple words, our model may remind of a dynamics of heat conduction in a \emph{foiled} composite material consisting of a media interlaced with 
very thin plates of different permeability. In material science such models are referred to as reinforced materials like
a glass wool reinforced by aluminium foil, as it was considered 
in Y\"uksel et al.\  \cite{yuksel2012effective}.

Such systems also appear in chemistry, biology and physics, see, e.g., Tanner \cite{Tanner78},
Dudko et al.\
\cite{dudko2004diffusion},
Moutal and Grebenkov
\cite{moutal2019diffusion},
Grebenkov
\cite{grebenkov2010pulsed},
{\'S}l{\k{e}}zak and Burov
\cite{slkezak2021diffusion}.

To formulate the model rigorously, we introduce a small parameter $\e>0$ and
assume that the state space $\bR\times\bR^n$, $n\geq 0$, is sliced into layers by countably many
hyperplane interfaces (membranes)
\ba
&\{a_k^\e\}\times \bR^n,\quad k\in\mathbb Z,
\ea
such that the family $\{a^\e_k\}\subseteq \bR$ has no accumulation points.  

We look for a $n+1$ dimensional continuous strong Markov process
$(X^\e,Y^\e)=(X^\e,Y^{\e,1},\dots,Y^{\e,n})$, $X^\e\in\bR$, $Y^\e\in \bR^n$ that behaves as follows.

Between the membranes $\{\{a_k^\e\}\times \bR^n\}_{k\in\mathbb Z}$
the process $(X^\e,Y^\e)$ is a usual diffusion with the drift vector $(b^i)_{0\leq i\leq n}$ and 
the diffusion matrix $(\sigma_l^i)_{0\leq i\leq n, 1\leq l\leq m}$, $m\geq 1$.
Upon hitting a membrane $\{a_k^\e\}\times \bR^n$ at a point
$(a_k^\e,y)$ at time $\tau$,
the process $(X^\e,Y^\e)$ starts anew and `leaves' the membrane in the direction $\pm(1,\theta(a_k^\e,y))$
with
the so-called penetration probabilities
$\frac12(1 \pm \e \beta(a_k^\e,y))$.

We look for such a process $(X^\e,Y^\e)$ as a solution of the
stochastic differential equation (SDE)
\ba
\label{e:XY}
X^\e_t&= x+\sum_{l=1}^m \int_0^t \sigma_{l}^0(X^\e_s,Y^\e_s)\,\di W^l_s + \int_0^t b^0(X^\e_s,Y_s^\e)\,\di s
+\sum_{k=-\infty}^\infty \e \int_0^t \beta(a_k^\e,Y_s^\e)\,\di L^{a_k^\e}_s (X^\e),\\
Y^{\e,i}_t&= y^i+\sum_{l=1}^m \int_0^t \sigma^{i}_l(X_s^\e,Y_s^\e)\,\di W^l_s +\int_0^t b^i(X_s^\e,Y_s^\e)\,\di s
+ \sum_{k=-\infty}^\infty \e\int_0^t  \beta(a_k^\e,Y_s^\e) \theta^i(a_k^\e,Y_s^\e)\,\di L^{a_k^\e}_s (X^\e),\\
&\quad i=1,\dots,n,
\ea
where $W=(W^1,\dots,W^m)$ is a standard $m$-dimensional Brownian motion and
$L^a(X^\e)$ is the symmetric local time of $X^\e$ at $a\in\bR$ (see the end of this Section for precise definitions).

The case $n=0$ corresponds to a one-dimensional diffusion $X^\e$ and point interfaces $\{a_k^\e\}$ located on the real line.

The goal of this paper is to establish the existence and uniqueness of a weak solution of the system \eqref{e:XY} and to study
its \emph{homogenization} limit as $\e\to 0$.

We make the following assumptions about the coefficients in the system \eqref{e:XY} and the location of the membranes. 
By $C^k_b$, $k\geq 1$, we denote the space of $k$ times continuously differentiable bounded functions with bounded derivatives.

\smallskip

\noindent
\textbf{Assumptions A:}

\smallskip
\noindent
\textbf{A}$_\text{coeff}$:
\ba
\beta, b^i, \theta^i, \sigma^i_l\in C^2_b(\bR\times\bR^n,\bR),\quad  i=0,\dots,n,\ l=1,\dots,m.
\ea

\smallskip
\noindent
\textbf{A}$_a$:
The points $\{a_k^\e\}$ satisfy
\ba\label{e:d}
a^\e_k=\int_0^{\e k} d(x)\,\di x,\quad k\in\mathbb Z,
\ea
where $d\in C^1_b(\bR,\bR)$ such that $\inf_{x\in\bR} d(x)>0$.

\smallskip
\noindent 
\textbf{A}$_{\Sigma}$:
The matrix $\Sigma=(\Sigma^{ij})_{0\leq i,j\leq n}$ with the entries
\ba
\Sigma^{ij}(x,y)=\sum_{l=1}^m \sigma^i_l(x,y)\sigma^j_l(x,y)
\ea
is uniformly positive definite.
 
The main result of this paper is formulated in the following theorems.

\begin{thm}
\label{t:exist}
Let Assumptions \emph{\textbf{A}} hold true. Then there is $\e_0\in(0,1]$ small enough such that for all $\e\in(0,\e_0]$
the system \eqref{e:XY} has a unique weak solution $(X^\e,Y^\e)$, which is a strong Markov process.
\end{thm}

\begin{rem}
The condition that the parameter $\e>0$ is small is essential. First, since the values $\frac12(1\pm\e\beta(x,y))$ have the meaning of penetration probabilities through the membranes, we have to demand that $\e\|\beta\|_\infty\leq 1$; otherwise the solution $(X^\e,Y^\e)$ does not exist even in the 
simplest one-dimensional case of a skew Brownian motion, see Harrison and Shepp \cite{HShepp-81}.   
Another condition on the smallness of $\e$ will be needed for the proof of 
uniqueness of the solution in a neighbourhood of each membrane $a^\e_k$, see Section
\ref{s:eu}. It is well known that uniqueness is necessary for the strong Markov property of the solution.
\end{rem}

\begin{thm}
\label{t:hom}
Let Assumptions \emph{\textbf{A}} hold true and let $(X^\e,Y^\e)$ be unique weak solutions of \eqref{e:XY}. In the limit $\e\to 0$ the weak convergence
\ba
(X^\e,Y^\e) \Rightarrow (X,Y)
\ea
in $C(\bR_+,\bR^{n+1})$ holds true, where $(X,Y)$ is a (unique) weak solution of the SDE
\ba
\label{e:XYlim}
X_t&= x+\sum_{l=1}^m \int_0^t \sigma_{l}^0(X_s,Y_s)\,\di W^l_s + \int_0^t \Big(b^0(X_s,Y_s)+\beta(X_s,Y_s) \frac{\Sigma^{00}(X_s,Y_s)}{d(X_s)}\Big)\,\di s ,\\
Y^{i}_t&= y^i+\sum_{l=1}^m \int_0^t \sigma^{i}_l(X_s,Y_s)\,\di W^l_s
+\int_0^t \Big(b^i(X_s,Y_s)   +\beta(X_s,Y_s) \theta^i(X_s,Y_s) \frac{ \Sigma^{00}(X_s,Y_s)}{d(X_s)}\Big)\,\di s.
&\quad i=1,\dots,n.
\ea
\end{thm}

As we see, the limiting process is a diffusion with the same diffusion matrix $\{\sigma^i_l\}$ as the processes $(X^\e,Y^\e)$. 
The semipermeable   membranes, i.e., the local time terms of \eqref{e:XY}, turn into the regular
additional drift which can alter diffusion's behaviour significantly as it is demonstrated in the following example.

\begin{exa}
Let $n=1$. We consider a two-dimensional stochastic differential equation
\ba
X^0_t&= x+W^1_t - \int_0^t Y_s^0\,\di s,\\
Y^0_t&= y+ W^2_t +\int_0^t X_s^0 \,\di s,
\ea
that describes, e.g., the velocity of a charged particle in a constant magnetic field subject to external stochastic electric field $(W^1,W^2)$, see
Chechkin et al.\
\cite{ChechkinGS02}. The random trajectory of a particle tends to make windings around 0 in the positive direction, see Fig.~\ref{f:2} (left).

Let us add to the model spatial membranes (lines) located equidistantly at $x$-points $a_{k}^\e=k\e$, $k\in\mathbb Z$, so that $d(x)\equiv 1$.
Assume that in some large ball around the origin the permeability characteristics $\beta$ and $\theta$ of the membranes are 
given by
\ba
\beta(x,y)&=\frac{2y^3}{\gamma+y^2},\\
\theta(x,y)&=-\frac{ x^3 y}{(\gamma^2+x^2)(\gamma+y^2)},\quad \gamma=10^{-2}.
\ea
Outside this ball we define them in such a way that they satisfy the assumptions \textbf{A}.
Then in the limit $\e\to 0$, the diffusion with local times $(X^\e,Y^\e)$ converges to a diffusion $(X,Y)$ which solves the SDE with 
additional \emph{interface-induced} drift:
\ba
X_t&= x+ W_t^1 - \int_0^t Y_s\,\di s
+ \int_0^t \frac{2Y^3_s  }{\gamma+Y^2 }\, \di s ,\\
Y_t&= y+ W^2_t +\int_0^t X_s \,\di s
- \int_0^t   \frac{2X_s^3Y_s^4 }{(\gamma^2+X^2_s)(\gamma+Y^2_s)^2 }    \,\di s.
\ea
In a large ball around the origin, the new drift approximately 
equals $(Y,-X)$, i.e., the rotation direction of the homogenized particle is now negative, see Fig.~\ref{f:2} (right). 
\begin{figure}
\begin{center}
 \includegraphics[width=.8\textwidth]{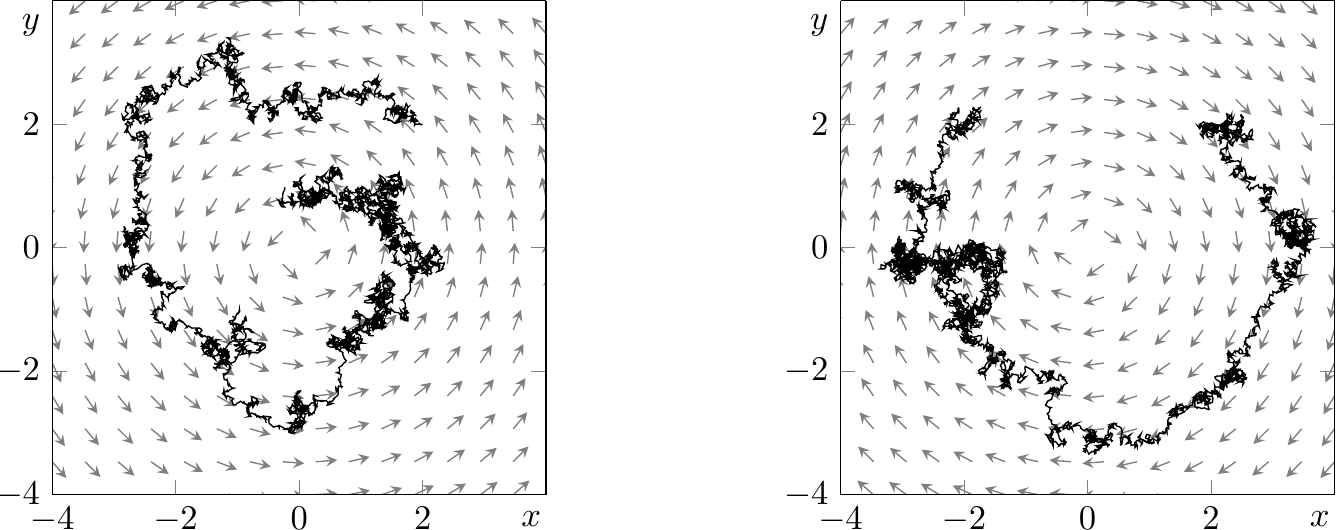}
\end{center}
\caption{Sample paths of the ``underlying'' diffusion $(X^0,Y^0)$ without membranes (left) and of the limiting homogenized diffusion $(X,Y)$ (right). The membranes with penetration probabilities $\frac12+\frac{\e}{2}\beta(x,y)$ and penetration directions $(1,\theta(x,y))$ generate an additional drift that reverses the diffusion's rotation direction. Both samples start at the point $(2,2)$.\label{f:2}}
\end{figure}
\end{exa}
The novelty of this paper, besides the stochastic homogenization Theorem \ref{t:hom}, consists in the proof of the existence and uniqueness
Theorem \ref{t:exist} in a \emph{multivariate} setting.

In dimension one, the first result on a existence and uniqueness for an SDE with local time was obtained by \cite{HShepp-81}. They proved that a skew Brownian motion is a unique strong solution to an SDE $X_t=x+W_t+(2\beta-1)L^0_t(X)$ for $|\beta|\leq 1$. In the proof, by using the speed measure the initial equation was transformed regularly into an equivalent one with the locally constant coefficients and without local times. Later, Le Gall
 \cite{legall1984one} generalized this approach and applied it to the general one-dimensional SDEs with local times of the unknown process. The resulting SDE has possibly discontinuous coefficients but does not involve local times. The general theory was outlined in 
 Engelbert and Schmidt \cite{engelbert1991strong}, see also Lejay \cite{Lejay-06}.

In higher dimensions, diffusion processes with reflection (i.e., with $\e\beta(x,y)=1$) are best studied. The existence problem for them is usually formulated in terms of a martingale problem (see, e.g., Stroock and Varadhan \cite{stroock1971diffusion}) or the Skorokhod problem (see, e.g., 
Tanaka
\cite{Tanaka1979}, and 
Lions and Sznitman
\cite{Lions_et1984}). Zaitseva \cite{Zaitseva2005} considered an SDE for a multidimensional Brownian motion with oblique skewness at a hyperplane 
provided that the skewing coefficient and the reflection direction are constant. Under this condition, the tangential coordinates do not depend 
on the normal one. This allows to obtain the strong existence and uniqueness of a solution by standard arguments.  
Unfortunately, this method is inapplicable in the case of oblique reflection since the coordinates are necessarily dependent. 

The general case of an oblique skew diffusion with non-constant coefficients is much more complicated and the question of existence and uniqueness is still open. 
Portenko \cite{portenko1976generalized,portenko1990generalized} formulated the problem of existence and uniqueness of a diffusion process with semipermeable membranes in term of a parabolic conjugation boundary problem. 
With the help of potential theory methods it was proved that the solution of such a problem
defines the semigroup of operators that corresponds to a process which diffusion characteristics exist in the sense of generalized functions. 
Portenko and Kopytko 
\cite{Portenko_et2012} obtained the same result for a
diffusion with a semipermeable membrane on a hyperplane with oblique reflection.
However it does not follow immediately from these works that the constructed process is a (unique) solution of an SDE with local time terms. 

We also mention the works by Ouknine et al.\ \cite{trutnau2015countably} and Ramirez \cite{ramirez2011multi} who 
investigated diffusions with infinitely many interfaces in dimensions one and two. 

Concerning the homogenization problem, besides the works cited in the beginning of the paper, 
the following authors studied homogenization of (regular) diffusions by analytic and stochastic methods:
Pardoux \cite{pardoux1999homogenization},
Pavliotis and Stuart \cite{pavliotis2008multiscale},
Hairer and Pardoux
\cite{hairer2008homogenization},
Makhno \cite{Makhno-12}.
%, Olla1994). 

Results on homogenization for diffusions with interfaces are rather sparse. Weinryb \cite{weinryb1984homogeneisation} 
studied the periodic homogenization of planar diffusions with permeable membranes on lines or circles under the assumption on existence and uniqueness of solution to the corresponding martingale problem. 
Hairer and Manson \cite{hairer2010one,hairer2011multi,hairer2010periodic} studied periodic diffusion homogenization with one interface. 
Limit theorems for one-dimensional diffusions with interfaces were obtained by Makhno \cite{makhno2016one,makhno2017diffusion}
and Krykun \cite{krykun2017convergence}.

\medskip
\noindent
\textbf{Sketch of the proof.}
The main idea of the proof is based on the decomposition of the solution $(X^\e,Y^\e)$ into segments that live on random time intervals between the 
subsequent hittings of the neighbouring membranes.
The existence and uniqueness result (Theorem \ref{t:exist}) for the system \eqref{e:XY} will be obtained with the help of a 
multivariate coordinate transformation $(U^\e,V^\e)=(F^\e(X^\e,Y^\e),G(X^\e,Y^\e))$ that transforms the system \eqref{e:XY}
into an equivalent system of SDEs with discontinuous coefficients \emph{without local times} in a neighbourhood of each membrane.
The uniqueness will follow from the result by Gao \cite{gao1993martingale}
and the global solution is obtained by glueing these solutions together (see Section \ref{s:eu}).

To prove the convergence Theorem \ref{t:hom} we note that the essential dynamics of the diffusion with local times $(X^\e, Y^\e)$ is catched by the embedded
Markov chain that comprises the values $(X^\e, Y^\e)$ at the hitting times of the membranes.  

To obtain the formula for the infinitesimal operator $\cL$ of the limit diffusion we study the pseudo-characteristic operator of the embedded Markov chain
\ba\label{eq:pseudo_generator}
\cL^\e f(x,y):=\frac{\E_{x,y} f(X^\e_{\tau^\e}, Y^\e_{\tau^\e})-f(x,y)}{\E_{x,y} \tau^\e},\quad x\in \{a_k^\e\}_{k\in\mathbb Z},\ y\in\bR^n,
\ea
where $\tau^\e$ is the first hitting time of a neighbouring membrane.

Let for brevity $n=1$. Writing the Taylor formula for a function $f\in C^3_b(\bR^2,\bR)$ we get
\ba\label{eq:pseudo_generator_Taylor}
\E_{x,y} f(X^\e_{\tau^\e}, Y^\e_{\tau^\e})-f(x,y)
&=f_x(x,y) \E_{x,y} (X^\e_{\tau^\e}-x)
+f_{y}(x,y) \E_{x,y} (Y^{\e}_{\tau^\e}-y)\\
&+\frac12 f_{xx}(x,y) \E_{x,y} (X^\e_{\tau^\e}-x)^2 
+f_{xy}(x,y) \E_{x,y} (X^\e_{\tau^\e}-x) (Y^{\e}_{\tau^\e}-y)\\
&+\frac12 f_{yy}(x,y) \E_{x,y} (Y^{\e}_{\tau^\e}-y)^2
+ \mathcal O\Big(\E_{x,y} |X^{\e}_{\tau^\e}-y|^3+\E_{x,y} |Y^{\e}_{\tau^\e}-y|^3 \Big)
\ea
We will be able to calculate the fine asymptotics of the average exit time from the current strip $\E \tau^\e$, which will be of the order $\e^2$
as well as the 
moments $\E_{x,y}(X_{\tau^\e}^\e-x)$, $\E_{x,y}(Y_{\tau^\e}^\e-y)$, etc, with accuracy $\mathcal O(\e^{3-\delta})$, $\delta\in (0,1)$ 
(see Lemma \ref{l:asympt} and 
Section \ref{s:Y}). 
These estimates will be derived by approximating the transformed diffusion without local times by a diffusion with 
coefficients
``frozen'' at the initial point $(x,y)$.

The limit second order generator $\cL$ that determines the limit diffusion $(X,Y)$, 
see equation \eqref{e:XYlim}, is a pointwise limit of the sequence $\{\cL^\e\}$ as $\e\to 0$.

The proof of the convergence Theorem \ref{t:hom}   consists of 
a) the standard step of establishing the weak relative compactness of the family $(X^\e,Y^\e)$ (Section \ref{s:wrc})
and b) of verification that any limit process $(X,Y)$ satisfies the (well-posed) martingale problem
\ba
\label{eq:686}
\E\Big[ f(X_t, Y_t)-f(X_s, Y_s)-\int_s^t \cL f(X_u, Y_u)\,\di u \Big|\rF_{s}\Big]=0,\quad f\in C^3_b(\bR^{n+1},\bR).
\ea

\medskip
\noindent
\textbf{Notation.} In this paper,
$|\cdot|$ denotes the Euclidiean distance in $\bR^n$, $n\geq 1$ as well as the Frobeinus norm of a matrix $A$, i.e., $|A|=(\sum_{i,j}a_{ij}^2)^{1/2}$;
$\|f\|_\infty=\sup_x|f(x)|$ is the supremum norm of a real-, vector- or matrix-valued function $f$; $x^+=\max\{x,0\}$. 
The Jacobian matrix of a vector-valued function $f$ is denoted by $Df$.
In particular, for a smooth $f$ we have $|f(x)-f(y)|\leq \|Df\|_\infty|x-y|$. 

Sometimes the constant $C>0$ denotes a generic constant that does not depend on $\e$; its value may vary within the same chain of inequalities. 

We write that $f(\e)\lesssim g(\e)$ if 
$|f(\e)|\leq C |g(\e)|$ for small $\e>0$. We also will write $f(\e)\lesssim \e^\infty$ if $f(\e)\lesssim \e^k$ for any $k\geq 1$.

We also recall that the symmetric local time $L^a(X)$ at $a\in\bR$ of a continuous real-valued semimartingale $X$ is the unique
non-decreasing process satisfying
\ba
|X_t-a|=|X_0-a|+\int_0^t \sgn (X_s-a)\,\di X_s+ L^a_t(X),
\ea
where
\ba
\sgn x=\begin{cases}
&         1,\quad x>0,\\
&         0,\quad x=0,\\
&         -1,\quad x<0,
        \end{cases}
\ea
see, e.g., Chapter VI in Revuz and Yor \cite{RevuzYor05}.
The local time $L^a(X)$ satisfies
\ba
L^a_t(X)=\lim_{\delta\to 0}\frac{1}{2\delta}\int_0^t \bI(|X_s-a|\leq \delta)\,\di \langle X\rangle_s.
\ea

\medskip

\noindent
\textbf{Acknowledgments:} O.A.\ and I.P.\ thank 
the German Research Council (grant Nr.\  PA 2123/7-1) and
the VolkswagenStiftung (grant Nr.\ 9B946) for financial support.
O.A.\ and A.P.\ were partially supported by the Alexander von Humboldt Foundation within the 
Research Group Linkage cooperation \emph{Singular diffusions: analytic and stochastic
approaches} between the University of Potsdam and the Institute of Mathematics of the National Academy of
Sciences of Ukraine.
 The authors are grateful to the anonymous referees for their helpful reports.

\section{Existence and uniqueness. Proof of Theorem \ref{t:exist}\label{s:eu}}

First, we consider the system \eqref{e:XY} in a neighbourhood of one membrane. Without loss of generality we fix $k=0$ so that
$a_0^\e\equiv 0$. Since outside the membrane $(X^\e,Y^\e)$ is a diffusion with regular coefficients, we assume that the initial values are $x=0$ and $y\in\bR^n$.
With some abuse of notation we denote $\beta(y):=\beta(0,y)$, $\theta(y):=\theta(0,y)$
and look for the solution $(X^\e,Y^\e)$ of the following SDE with one membrane:
\ba
\label{e:XY0}
X^\e_t&= \sum_{l=1}^m \int_0^t \sigma_{l}^0(X^\e_s,Y^\e_s)\,\di W^l_s + \int_0^t b^0(X^\e_s,Y_s^\e)\,\di s
+\e \int_0^t \beta(Y_s^\e)\,\di L^0_s (X^\e),\\
Y^{\e,i}_t&= y^i+\sum_{l=1}^m \int_0^t \sigma^{i}_l(X_s^\e,Y_s^\e)\,\di W^l_s +\int_0^t b^i(X_s^\e,Y_s^\e)\,\di s
+ \e\int_0^t  \beta(Y_s^\e) \theta^i(Y_s^\e)\,\di L^0_s (X^\e),\quad i=1,\dots,n.
\ea

We construct the transformation of $(X^\e,Y^\e)$ into a diffusion without the local time terms. Let $\e>0$ 
be small enough such that $\e\|\beta\|_\infty\leq 1/2$, 
and let
\ba
\label{e:BB}
B^\e(y):=\frac{1-\e\beta(y)}{1+\e\beta(y)},\quad y\in\bR^n. 
\ea
We have
\begin{align}
\label{e:B}
|B^\e(y)-1| &\leq 4\e\|\beta\|_\infty,\\
\label{e:nablaB}
\|\nabla B^\e\|_\infty & \leq 8\e \|\nabla \beta\|_\infty.
\end{align}
Define the functions
\ba
\label{e:FG}
F^\e(x,y)&=x\bI_{(-\infty,0)}(x)+x B^\e( y)\bI_{[0,\infty)}(x)=x + x^+ (B^\e(y)-1);\\
G(x,y)&:=y-x\theta(y).
\ea

It is clear that the transformation $(x,y)\mapsto (F^\e(x,y),G(x,y))$ does not have to be a one-to-one bijection of $\bR \times\bR^n$.
However, this map is injection  in a ``thin'' strip $[-A\e, A\e]\times \bR^n$, \ for any $A>0$ fixed and $\e>0$ small enough.

The proof will consist in the application of the Hadamard's global inverse function theorem, see Theorem 6.2.4 in Krantz and Parks \cite{krantz2002implicit}.

\begin{thm}[Hadamard]
\label{t:Hadamard}
Let $N\geq 1$, $H\colon \bR^N\to \bR^N$,  be a $C^2$ mapping. Suppose that $H(0) = 0$
and that the Jacobian determinant of $H$ is nonzero at each point. Finally, assume
that  
\ba
\|(DH)^{-1}(\cdot)\|_\infty <\infty.
\ea
 Then $H$ is a $C^2$ diffeomorphism.
\end{thm}

\begin{lem}
For any $A>0$ fixed and $\e>0$ small enough, the transformation $(F^\e,G)|_{[-A\e,A\e]\times \bR^n}$ can be continued to a Lipschitz continuous
homeomorphism $\bR\times\bR^n\to \bR\times\bR^n$
which is a $C^2$-diffeomorphism on each of the half-spaces $\bR_-\times\bR^n$ and $\bR_+\times\bR^n$.
Moreover, for $\e>0$ small enough
\ba
\label{e:incl}
{} [-A'\e, A'\e] \times \bR^n \subseteq(F^\e,G)([-A\e,A\e]\times \bR^n)\subseteq [- A''\e,A''\e]\times \bR^n
\ea
for some $0<A'\leq A''$.
\end{lem}
\begin{proof}
Let us extend the mapping $(F^\e,G)$ defined in \eqref{e:FG} from $(0,A\e)\times \bR^d$ to $\bR\times \bR^d$ as follows.
Let $g_\e\in C^\infty(\bR,\bR)$ be such that $g_\e(x)=x$ for $|x|\leq A\e$, 
$g_\e(x)=0$ for $|x|\geq 2A\e$ and
\ba
\label{eq:255}
\|g_\e\|_\infty\leq 2A\e,\quad  \|g'_\e\|\leq 2A.
\ea
For each $\e>0$, consider the function $H^\e:=(\widetilde F^\e, \widetilde G^\e)$
where
\ba
\widetilde F^\e(x,y)&=x+g_\e(x) (B^\e(y)-1),\\
\widetilde G^\e(x,y)&=y-g_\e(x)\theta(y).
\ea
Then $H^\e$ is $C^2$ function that coincides with $( F^\e,  G)$ on $(0,A\e)\times \bR^n$ and 
is the identity map on $(\bR \setminus(-2A\e,2A\e))\times \bR^d$. Moreover we have 
that 
\ba
&H^\e([0,\infty)\times \bR^n)\subseteq [0,\infty)\times \bR^n,\\
&H^\e((-\infty,0)\times \bR^n)\subseteq (-\infty,0)\times \bR^n
\ea
if $\e$ is sufficiently small. Note that $\sup_{x,y} |D H^\e(x,y)|<\infty$.  Hence, Hadamard's theorem will yield  that
$H$ is a $C^2$-diffeomorphism of $(0,\infty)\times \bR^n$ if we show that 
$\inf_{x,y}\det D H^\e(x,y)>0$. This follows, however, from \eqref{e:B}, \eqref{e:nablaB}, \eqref{eq:255},  and the observation that
\ba
D H^\e(x,y)
=\begin{pmatrix}
       1+ g'_\e(x) (B^\e(y)-1) & g_\e(x) D B^\e (y)\\
        - g'_\e(x) \theta(y) & \text{Id}- g_\e(x)D \theta(y)\\
       \end{pmatrix}.
\ea
The inclusion 
\ba
H^\e([0,A\e]\times \bR^n)\subseteq [0 ,A''\e]\times \bR^n
\ea
is obvious for $A''$ large enough.

Let us verify the inclusion
\ba
\label{eq:incl_diff}
[0, A'\e] \times \bR^n \subseteq H^\e([0,A\e]\times \bR^n)
\ea
for small $A'>0$. 

Notice that if $x>A\e$, then $\widetilde F^\e(x,y)>A\e (1-2\e \|\beta\|_\infty)$. Hence 
\ba
{}[0,A\e (1-2\e \|\beta\|_\infty)]\times \bR^n  \cap H^\e([ A\e, \infty)\times \bR^n)=\emptyset.
\ea
Since $H^\e([0,\infty)\times \bR^n)= [0,\infty)\times \bR^n$, we have 
\eqref{eq:incl_diff} with $A'=A/2$ if $\e \|\beta\|_\infty<1/4.$

The extension of $(F^\e,G)$ to the left half-space can is constructed similarly.
\end{proof}

We denote
\ba
U^\e&:=F^\e(X^\e, Y^\e), \\
V^\e&:=G(X^\e, Y^\e).\\
\ea
Let $(\Phi^\e,\Psi^\e):=(F^\e,G)^{-1}$ be the %$C^2$
inverse mapping in some strip $[-A\e,A\e]\times \bR^n$, i.e.,
\ba
\label{e:PP}
X^\e&=\Phi^\e(U^\e,V^\e),\\
Y^\e&=\Psi^\e(U^\e,V^\e).
\ea
We set $\tau^\e:=\inf\{t\geq 0\colon X^\e_t\notin (-A\e,A\e)\}$ and  $\tilde\tau^\e:=\inf\{t\geq 0\colon (U^\e_t, V^\e_t)\notin (F^\e, G)\left((-A\e,A\e)\times \bR^n\right)\}$.

\begin{prp}
The process $(X^\e,Y^\e)_{t\leq \tau^\e}$ is a (weak) solution of \eqref{e:XY0} if and only if $(U^\e,V^\e)_{t\leq \tilde\tau^\e}$ is a (weak) solution of
\ba
\label{e:UV0}
U^\e_t
&= \int_0^t   \Big[ b^0(\cdot,\cdot)+ \phi^{\e, 0}(\cdot,\cdot) \Big]\circ \Big(\Phi^\e(U_s^\e,V_s^\e),\Psi^\e(U_s^\e,V_s^\e)\Big) \,\di s\\
&+\sum_{l=1}^m\int_0^t   \Big[ \sigma_{l}^0 (\cdot,\cdot) + \phi_{l}^{\e,0}(\cdot,\cdot)     \Big]
\circ \Big(\Phi^\e(U_s^\e,V_s^\e),\Psi^\e(U_s^\e,V_s^\e)\Big)  \,\di W^l_s,   \\
V^{\e,i}_t&=y^i
+\int_0^t \Big[ b^i(\cdot,\cdot)-\theta^i(\cdot)b^0(\cdot,\cdot))
-  \sum_{j=1}^n  \theta^i_{y^j}(\cdot)\Sigma^{i0}(\cdot,\cdot) +\psi^i(\cdot,\cdot)\Big] \circ\Big(\Phi^\e(U_s^\e,V_s^\e),\Psi^\e(U_s^\e,V_s^\e)\Big)  \,\di s\\
&+\sum_{l=1}^m \int_0^t \Big[  \sigma_{l}^i(\cdot,\cdot)
- \theta^i(\cdot)\sigma_{l}^0(\cdot,\cdot)
 -\psi_{l}^i(\cdot,\cdot)  \Big] \circ \Big(\Phi^\e(U_s^\e,V_s^\e),\Psi^\e(U_s^\e,V_s^\e)\Big)   \,\di W^l_s, \\
\ea
where
\ba
\phi^{\e,0}(x,y)&=
(B^\e(y)-1)  b^0(x,y) \bI(x> 0)
+ \sum_{i=1}^n B_{y^i}^\e(y) \Big(  x^+b^i(x,y)+   \Sigma^{0i}(x,y)\bI(x> 0)\Big)
\\+ & \frac12 \sum_{i,j=1}^n x^+B^\e_{y^iy^j}(y)\Sigma^{ij}(x,y),
\\
\phi_{l}^{\e,0}(x,y)
&=(B^\e(y)-1)\bI(x> 0)  \sigma_{l}^0(x,y) + x^+ \sum_{i=1}^n B^\e_{y^i}(y) \sigma_{l}^i(x,y),\\
\psi^{i}(x,y)
&=-  \sum_{j=1}^n x\theta_{y^j}^i(y)   b^j(x,y)
-\frac12 \sum_{j,k=1}^n x\theta^i_{y^jy^k}(y)\Sigma^{jk}(x,y),\\
\psi_{l}^{i}(x,y)&=-\sum_{j=1}^n x\theta^i_{y^j}(y)\sigma_{l}^j(x,y),\quad i=1,\dots,n,\ l=1,\dots,m.
\ea
\end{prp}

\begin{proof}
Recall the Tanaka formula
\ba
(X^\e_t)^+=\int_0^t \Big(\bI(X^\e_s>0)+ \frac12 \bI(X^\e_s=0)\Big) \,\di X^\e_s + \frac12 L^0_t(X^\e).
\ea
Hence the application of 
 %We apply the Tanaka formula to $(X^\e)^+$,
 the It\^o formula to $B^\e(Y^\e)-1$ and  the product It\^o formula yields  that
 \ba
\label{e:UVXY}
U^\e_t
&= \int_0^t   \Big(  b^0+ \phi^{\e,0}\Big)(X^\e_s,Y^\e_s) \,\di s
+\sum_{l=1}^m\int_0^t   \Big( \sigma_{l}^0  + \phi_{l}^{\e,0}    \Big)(X^\e_s,Y^\e_s)\,\di W^l_s,   \\
V^{\e, i}_t&=y^i
+\int_0^t \Big( b^i -\theta^i b^0 - \sum_{j=1}^n  \theta^i_{y^j}\Sigma^{0j} + \psi^{i}\Big)(X_s^\e,Y_s^\e) \,\di s\\
&+\sum_{l=1}^m \int_0^t \Big(  \sigma_{l}^i - \theta^i \sigma_{l}^0 +\psi_{l}^{i}  \Big)(X_s^\e,Y_s^\e)\,\di W^l_s, \\
\ea
and the representation \eqref{e:UV0} follows immediately.
Note that all the coefficients of the system \eqref{e:UV0} are smooth on the half-spaces $\bR_\pm\times \bR^n$ and is discontinuous on the
hyperplane $U=0$. To transform the system \eqref{e:UV0} into \eqref{e:XY0} we apply the It\^o formula with local times as proven by Peskir
\cite{peskir2007change}.
\end{proof}

\begin{rem}
\label{r:small}
For $\e$ small enough,
the functions $\phi^{\e,0}$, $\phi^{\e,0}_l$, $\psi^{i}$, $\psi^{i}_l$ satisfy
\ba
&|\phi^{\e,0}|,|\phi^{\e,0}_l|\leq C\e(1+|x|),\\
&|\psi^{i}|,|\psi^{i}_l|\leq C|x|.
\ea
In particular, in a strip $[-A\e,A\e]\times \bR^n$, $A>0$, they are bounded by $C\e$ for some $C>0$.
\end{rem}

\noindent
\emph{Proof of Theorem \ref{t:exist}}.
We consider the system \eqref{e:UV0} in some strip $[-A''\e,A''\e]\times \bR^n$ 
extend all the coefficients in \eqref{e:UV0} such that they are bounded, smooth functions on the half-spaces $\{ u<0\}$ and $\{ u>0\}$
that may have a discontinuity on $\{u=0\}$. Note that the diffusion matrix of $(U^\e, V^\e)$ is bounded, uniformly elliptic on $\bR\times\bR^n$, 
continuous on the 
half-spaces $\{ u<0\}$ and $\{ u>0\}$ but may have a discontinuity on the hyperplane $\{0\}\times\bR^n$. 
Hence the system \eqref{e:UV0} has a weak solution by Krylov \cite{krylov1969ito}. In dimension one, i.e., for $n=0$, uniqueness follows
by means of the theorems by Nakao \cite{nakao1972pathwise} and Yamada and Watanabe \cite{yamada1971uniqueness}. In dimension two, i.e., for 
$n=1$, uniqueness follows again from Krylov \cite{krylov1969ito}. For $n\geq 2$, uniqueness follows from Theorem 1.1 in Gao \cite{gao1993martingale}.
Hence, this solution is also uniquely defined up to the exit from  the strip and this solution is independent 
of extension of coefficients to the whole space.
Since all the coefficients are bounded, the process $V^\e$ cannot explode in finite time with probability 1.
Hence the process $(X^\e,Y^\e)$ is also uniquely defined up to the moment of exiting from some strip around the membrane.
A global solution is obtained by gluing together the solutions in each strip.
By uniqueness, this solution is strong Markov.

Eventually we have to check that the solution does not blow up in a finite time. This will be shown in Lemma \ref{lem:compactCOnt} later.
\hfill$\Box$

\section{Dynamics around one membrane}

We consider the dynamics of the process $(X^\e,Y^\e)$ inside the strips $(a_{k-1}^\e,a_{k+1}^\e)$, $k\in\mathbb Z$.
For definiteness we assume in this section that $k=0$, $a_0^\e=0$ and
denote $a^\e_\pm=\pm a_{\pm 1}^\e >0$, $d:=d(0)$, $d':=d'(0)$.
Moreover we omit the argument $x=0$ in all functions, i.e., we write $\beta(y):=\beta(0,y)$, $b^i(y):=b^i(0,y)$ etc.

The aim of this section is to obtain accurate moment estimates for 
expectations that appear in  \eqref{eq:pseudo_generator_Taylor}. This will allow us to show that the operator $\cL^\e$ defined in \eqref{eq:pseudo_generator} approximates in some sense generator of the limit diffusion.

Due to assumption $\mathbf{A}_a$, the following expansion holds true:
\ba
\label{e:a}
a^\e_+&=d\e + \frac12 d'\e^2+ r^+(\e),\\
a^\e_- &=d\e - \frac12 d'\e^2+ r^-(\e),\\
\ea
where the $|r^\pm(\e)|\lesssim\e^3$ (uniformly over all membranes) and $\e>0$ small enough.
In particular,
\ba
a^\e_++a_-^\e=2 d\e + \mathcal O(\e^3),
\ea
Setting 
\ba
\label{eq:widehatA1}
\widecheck A=\frac12\inf_{x}d(x) \quad \text{and}\quad   \widehat A=2\sup_x d(x) 
\ea
we can  conclude that for small $\e>0$ 
\ba\label{eq:widehatA2}
{}[-\widecheck A\e,\widecheck A\e]\times \bR^n\subseteq [a_-^\e,a_+^\e]\times \bR^n\subseteq [-\widehat A\e,\widehat A\e]\times \bR^n.
\ea

\subsection{Properties of the coordinate transformation $(\Phi^\e,\Psi^\e)$}

Let $(\Phi^\e,\Psi^\e)$ be the inverse mapping to $(F^\e,G)$ defined in \eqref{e:FG}.

\begin{lem}
\label{l:estPhiPsi}
Let $A>0$ be fixed. For $\e>0$ small enough the following first order expansion holds true in a strip $[-A\e,A\e]\times \bR^n$:
\begin{align}
\label{e:Phi}
\Phi^\e(u,v)&=\begin{cases}
   u,\quad  (u,v)\in[-A\e,0]\times \bR^n,\\
   u \Big(1+2\e \beta(v)\Big)+ h_1^\e(u,v),\quad (u,v)\in(0,A\e]\times\bR^n,\\
              \end{cases}\\
\label{e:Psi}
\Psi^\e(u,v)&=v+\theta(v)u+h_2^\e(u,v),\quad (u,v)\in[-A\e,A\e]\times \bR^n
\end{align}
where 
\begin{align}
|h_1^\e(u,v)|&\lesssim \e^3 ,\\
|h_2^\e(u,v)|&\lesssim \e^2
\end{align}
uniformly in the strip $(u,v)\in[-A\e,A\e]\times \bR^n$.
\end{lem}
\begin{proof}
1. Let $u\in [-A\e,0]$. Then
\ba
u&=x,\\
v&=y-\theta(y)x.
\ea
Hence $\Psi^\e=\Psi$ is a (unique) solution of the equation $\Psi-\theta(\Psi)u-v=0$. This yields
\ba
|\Psi-v|\leq \|\theta\|_\infty |u|
\ea
and
\ba
\Psi-v-\theta(v)u
&= \Big(\theta(\Psi)-  \theta(v)\Big)u.
\ea
So
\ba
|\Psi-v-\theta(v)u|
&\leq \|D\theta\|_\infty |\Psi-v| |u|\leq  \|D\theta\|_\infty\|\theta\|_\infty u^2 \lesssim \e^2.
\ea

\noindent
2. Analogously, let $u>0$. Then
$u= x B^\e(y)$ and $\Psi^\e$ satisfies
\ba
\Psi^\e-\frac{\theta(\Psi^\e)}{B^\e(\Psi^\e)}u-v=0.
\ea
Hence
\ba
|\Psi^\e-v|\leq \Big\|\frac{\theta }{B^\e }\Big\|_\infty |u|
\ea
and
\ba
\Big|\Psi^\e-v-\frac{\theta(v)}{B^\e(v)}u\Big|
& \leq   \Big|\frac{\theta(\Psi^\e)}{B^\e(\Psi^\e)}-\frac{\theta(v)}{B^\e(v)} \Big| \cdot  |u|
\leq  \Big\|D\Big(\frac{\theta}{B^\e}\Big)\Big\|_\infty\cdot \Big\|\frac{\theta}{B^\e}\Big\|_\infty\cdot  u^2.
\ea
Taking into account \eqref{e:BB}, \eqref{e:B} and \eqref{e:nablaB} we get that for $\e>0$ small enough $\|\theta/B^\e\|_\infty \leq 2\|\theta\|_\infty$
and $\|D(\theta/B^\e)\|_\infty\leq 2 \|D\theta\|_\infty+\|\theta\|_\infty$. Since $|1/B^\e(v)- 1|\leq 4\e\|\beta\|_\infty$
we have
\ba
\Big|\Psi^\e-v-\theta(v)u \Big|
\leq \Big|\Psi^\e-v-\frac{\theta(v)}{B^\e(v)}u\Big|+\Big|\Big(\frac{1}{B^\e(v)}-1\Big)\theta(v)u\Big|
\leq C(u^2+|u|\e)\lesssim \e^2.
\ea
Using the previous estimate we obtain
\ba
\Big|\Phi^\e- u (1+2\e \beta(v))  \Big|&=|u|\Big| \frac{1}{B^\e(\Psi^\e)}- 1-2\e\beta(\Psi^\e)\Big|
+ 2\e |u| |\beta(\Psi^\e)-\beta(v)|\\
&\leq 3|u|\e^2\|\beta\|_\infty +2\e |u|\|\nabla \beta \|_\infty|\theta(v)u + h_2^\e(u,v)|\lesssim \e^3. 
\ea
\end{proof}

\subsection{Rough estimates of the exit time from the strip}

Let
\ba
\tau^\e=\inf\{t\geq 0\colon X_t^\e\notin (-a^\e_-,a_+^\e)\}.
\ea

To study the exit from the strip, we use the transformation \eqref{e:UVXY}
where $U^\e,V^\e$ are It\^o processes without the local time terms.

We extend the processes $(U^\e_t,V^\e_t)_{t\leq \tau^\e}$ obtained in \eqref{e:UVXY} to $t\geq 0$ by setting
\ba
\label{e:U}
U^\e_t
&= \int_0^t   \Big(  b^0+ \phi^{\e,0}\Big)(X^\e_{s\wedge\tau^\e},Y^\e_{s\wedge\tau^\e}) \,\di s
+\sum_{l=1}^m\int_0^t   \Big( \sigma_{l}^0  + \phi_{l}^{\e,0}    \Big)(X^\e_{s\wedge\tau^\e},Y^\e_{s\wedge\tau^\e})\,\di W^l_s   \\
V^{\e,i}_t&=y^i
+\int_0^t \Big( b^i -\theta^i b^0 - \sum_{j=1}^n  \theta^i_{y^j}\Sigma^{0j} + \psi^{i}\Big)(X^\e_{s\wedge\tau^\e},Y^\e_{s\wedge\tau^\e}) \,\di s\\
&+\sum_{l=1}^m \int_0^t \Big(  \sigma_{l}^i - \theta^i \sigma_{l}^0 +\psi_{l}^{i}  \Big)(X^\e_{s\wedge\tau^\e},Y^\e_{s\wedge\tau^\e})\,\di W^l_s. \\
\ea

With the help of \eqref{e:Phi} we have
\ba
&X_{\tau^\e}^\e=-a^\e_-\quad \Leftrightarrow\quad  U_{\tau^\e}^\e=-a^\e_-,\\
&X_{\tau^\e}^\e=a^\e_+\quad \Leftrightarrow\quad  a^\e_+=\Phi^\e(U_{\tau^\e}^\e,V_{\tau^\e}^\e)=U^\e_{\tau^\e}(1+2\e\beta(V^\e_{\tau^\e}))
+h_1^\e(U^\e_{\tau^\e},V^\e_{\tau^\e}).
\ea
Recall constants $\widecheck A$, $
\widehat A$ from \eqref{eq:widehatA1}, \eqref{eq:widehatA2}, and define the stopping times 
\ba
\label{e:taus}
\widecheck\tau^\e&=\inf\Big\{t\geq 0\colon U_t^\e\notin (-\widecheck A\e,\widecheck A \e)\Big\},\\
\widehat\tau^\e&=\inf\Big\{t\geq 0\colon U_t^\e\notin (-\widehat A\e,\widehat A\e)\Big\} 
\ea
such that
\ba
\label{e:tauss}
\widecheck\tau^\e\leq \tau^\e\leq \widehat \tau^\e.
\ea

\begin{lem}
\label{l:tau}
There are $\gamma, A\in (0,\infty)$ such that for all $\e>0$ small enough
\ba
\sup_{y\in\bR^n}\E_{0,y}\ex^{\gamma \e^{-2}\tau^\e}\leq A.
\ea
Consequently,
\ba
\sup_{y\in\bR^n} \E_{0,y} (\tau^\e)^k \lesssim \e^{2k},\quad k\geq 1.
\ea
Moreover,
\ba
\label{eq:estimate_exit_below}
\inf_{y\in\bR^n} \E_{0,y} \tau^\e \gtrsim\e^2.
\ea
\end{lem}
\begin{proof} 
1. We use \eqref{e:tauss} and
show that $\E_{0,y}\ex^{\gamma \widehat\tau^\e/\e^2}\leq A$ for all $\e>0$ small.
Consider the Lyapunov function
\ba
h^\e(t,x)=\ex^{\frac{\gamma t}{\e^2}}(\alpha \e^2-x^2).
\ea
For $\alpha >\widehat A^2$, $h^\e(t,x)\geq 0$ on $x\in[-\widehat A\e,\widehat A\e]$.

We have
\ba
\partial_t h^\e(t,x)&= \frac{\gamma}{\e^2}  \ex^{\frac{\gamma t}{\e^2}}(\alpha \e^2-x^2),\\
\partial_x h^\e(t,x)&= -2 \ex^{\frac{\gamma t}{\e^2}}x,\\
\partial_{xx} h^\e(t,x)&= -2 \ex^{\frac{\gamma t}{\e^2}}.
\ea
The It\^o formula yields for every $N\geq 1$:
\ba
\E_{0,y} h(\widehat\tau^\e\wedge N,U^\e_{\widehat\tau^\e\wedge N})
&=\alpha \e^2
+\E_{0,y}\int_0^{\widehat\tau^\e\wedge N} \frac{\gamma}{\e^2}  \ex^{\frac{\gamma s}{\e^2}}(\alpha\e^2-|U^\e_s|^2)\,\di s\\
&-\E_{0,y}\int_0^{\widehat\tau^\e\wedge N}  2 \ex^{\frac{\gamma s}{\e^2}}U_s^\e
\Big(b^0(X_{s\wedge\tau^\e}^\e,Y_{s\wedge\tau^\e}^\e)+\phi^{\e,0}(X_{s\wedge\tau^\e}^\e,Y_{s\wedge\tau^\e}^\e )\Big)\,\di s\\
&- \sum_{l=1}^m \E_{0,y} \int_0^{\widehat\tau^\e\wedge N}  \ex^{\frac{\gamma s}{\e^2}}
\Big( \sigma_{l}^0( X_{s\wedge\tau^\e}^\e,Y_{s\wedge\tau^\e}^\e ) + \phi^{\e,0}_l(X_{s\wedge\tau^\e}^\e,Y_{s\wedge\tau^\e}^\e  )    \Big)^2\,\di s.\\
\ea
By Assumption \textbf{A}$_\Sigma$  and Remark \ref{r:small} there are $\delta_1,\delta_2\in (0,\infty)$ such that
\ba
&\sum_{l=1}^m (\sigma_{l}^0(x,y)+ \phi^{\e,0}_l(x,y) )^2\geq \delta_1,\\
&|b^0(x,y)+\phi^{\e,0}(x,y)|\leq \delta_2\quad \text{ on } (x,y)\in [-\widehat A\e,\widehat A\e]\times \bR^n.%,\\
%&|U^\e_t|\leq C\e\quad \text{ on } t\leq \widehat\tau^\e.
\ea
Since $|U^\e_s|\leq \widehat A\e$ on $s\leq \widehat\tau^\e$, we have
\ba
\frac{\gamma}{\e^2}  (\alpha\e^2-|U^\e_s|^2)
&-2 U_s^\e (b^0(X_{s\wedge\tau^\e}^\e,Y_{s\wedge\tau^\e}^\e)+\phi^{\e,0}(X_{s\wedge\tau^\e}^\e,Y_{s\wedge\tau^\e}^\e))- \sum_{l=1}^m \Big( \sigma_{l}^0(X_{s\wedge\tau^\e}^\e,Y_{s\wedge\tau^\e}^\e) + \phi_{l}^{\e,0}(X_{s\wedge\tau^\e}^\e,Y_{s\wedge\tau^\e}^\e)    \Big)^2 \\
&\leq \gamma \alpha
+\widehat A \e \delta_2 - \delta_1  \leq -\frac{\delta_1}{2}
\ea
for some $\gamma>0$ small and all $\e>0$ small.

Hence
\ba
0&\leq \alpha \e^2 -\frac{\delta_1}{2}\E_{0,y}\int_0^{\widehat\tau^\e\wedge N} \ex^{\frac{\gamma s}{\e^2}}\,\di s
\leq \alpha \e^2 - \frac{\delta_1}{2}   \frac{\e^2}{\gamma} \E_{0,y} \ex^{\frac{\gamma (\widehat\tau^\e\wedge N)}{\e^2}} + \frac{\delta_1}{2}   \frac{\e^2}{\gamma},\\
\E_{0,y} \ex^{\frac{\gamma (\widehat\tau^\e\wedge N)}{\e^2}}&\leq \frac{2\alpha \gamma}{\delta_1}+1=:A,\quad N\geq 1,\\
\ea
and the statement follows for $\widehat\tau^\e$ by the monotone convergence theorem.
Using that $x^k\leq \ex^x k!$ we get
\ba
&\E_{0,y} (\widehat\tau^\e)^k \leq A \frac{k!}{\gamma^k}\e^{2k}\lesssim \e^{2k}.  
\ea
2. Applying the It\^o formula to $|U^\e|^2$ yields
\ba
\widecheck A\e^2\leq \E_{0,y} |U^\e_{\widecheck\tau^\e}|^2
&= 2\E_{0,y} \int_0^{\widecheck\tau^\e} U^\e_s
\Big( b^0(X_{s\wedge\tau^\e}^\e,Y_{s\wedge\tau^\e}^\e)+\phi^{\e,0}(X_{s\wedge\tau^\e}^\e,Y_{s\wedge\tau^\e}^\e )  \Big) \,\di s \\
&+ \sum_{l=1}^m\E_{0,y}\int_0^{\widecheck\tau^\e}
\Big( \sigma_{l}^0( X_{s\wedge\tau^\e}^\e,Y_{s\wedge\tau^\e}^\e ) + \phi_{l}^{\e, 0}(X_{s\wedge\tau^\e}^\e,Y_{s\wedge\tau^\e}^\e  )    \Big)^2 \,\di s \\
&\leq 2\widecheck A \e\delta_2 \E_{0,y}\widecheck\tau^\e +    3\sum_{l=1}^m\|\sigma_{l}^0 \|^2_\infty \cdot \E_{0,y}\widecheck\tau^\e
\ea
and the lower bound for $\E_{0,y}\tau^\e$ follows.
\end{proof}

\begin{lem}
\label{l:V}
For each $k\geq 1$ and $\e>0$ small enough
 \ba
&\E_{0,y} \sup_{t\leq \tau^\e } |V^\e_t-y |^{2k}\lesssim \e^{2k},\\
&\E_{0,y} \sup_{t\leq \tau^\e } |Y^\e_t-y |^{2k}\lesssim \e^{2k}.\\
\ea
In particular, for any $\delta\in (0,1)$ and any $n\geq 1$
\ba
&\P\Big(  \sup_{t\leq \tau^\e } |V^\e_t-y|\geq \e^{1-\delta}  \Big)\lesssim \e^{\infty},\\
&\P\Big(  \sup_{t\leq \tau^\e } |Y^\e_t-y|\geq \e^{1-\delta}  \Big)\lesssim \e^{\infty}.
\ea

\end{lem}
\begin{proof}
Recall that
\ba
V^{\e,i}_t=y^i
&+\int_0^t \Big( b^i -\theta^i b^0 - \sum_{j=1}^n  \theta^i_{y^j}\Sigma^{0j} + \psi^i\Big)(X_s^\e,Y_s^\e) \,\di s\\
&+\sum_{l=1}^m \int_0^t \Big(  \sigma_{l}^i - \theta^i \sigma_{l}^0 +\psi_{l}^{i}  \Big)(X_s^\e,Y_s^\e)\,\di W^l_s.
\ea
With the help of the Doob inequality and \eqref{e:UVXY}, we estimate for each $i=1,\dots,n$ and $k\geq 1$: %Hence, for $k\geq 1$
\ba
\E_{0,y}&\sup_{t\leq \widehat\tau^\e} (V^{\e,i}_t-y^i)^{2k}\\
&\leq  (m+1)^{2k-1}\E_{0,y}\sup_{t\leq \widehat\tau^\e}
\Big|\int_0^t   \Big( b^i -\theta^i b^0 - \sum_{j=1}^n  \theta^i_{y^j}\Sigma^{0j} + \psi^{i}\Big)(X_s^\e,Y_s^\e)    \,\di s\Big|^{2k} \\
& + (m+1)^{2k-1} \sum_{l=1}^m \E_{0,y}\sup_{t\leq \widehat\tau^\e}
\Big|\int_0^t  \Big(  \sigma_{l}^i - \theta^i \sigma_{l}^0 +\psi_{l}^{i}  \Big)(X_s^\e,Y_s^\e)      \,\di W^l_s\Big|^{2k} 
\\
&\leq
C_1 \E\Big[ (\widehat\tau^\e)^{2k-1} \int_0^{\widehat\tau^\e}
\Big(( b^i -\theta^i b^0 - \sum_{j=1}^n  \theta^i_{y^j}\Sigma^{0j} + \psi^{i})(X_s^\e,Y_s^\e)   \Big)^{2k}\,\di s\Big]   \\
&+ 
C_2 \sum_{l=1}^m \E \Big[\int_0^{\widehat\tau^\e} \Big((\sigma_{l}^i - \theta^i \sigma_{l}^0 +\psi_{l}^{i} )(X_s^\e,Y_s^\e) \Big)^2\,\di s\Big]^k\\
&\leq C_3\E_{0,y}\Big[(\widehat\tau^\e)^k +(\widehat\tau^\e)^{2k}\Big] \leq C_4\e^{2k}.
\ea
By Markov's inequality, for any $k\geq 1$ and $\delta\in (0,1)$
\ba
\P_{0,y}\Big(  \sup_{t\leq \tau^\e } |V^\e_t-y|\geq \e^{1-\delta}  \Big)
&=\P_{0,y}\Big(  \sup_{t\leq \tau^\e }|V^\e_t-y|^{2k}\geq \e^{2k-2k\delta}  \Big)\\
&\leq \e^{-2k+2k \delta}\E_{0,y}\sup_{t\leq \tau^\e} |V^\e(t)-y|^{2k}
\leq C_5 \e^{2k\delta}.
\ea
Choosing $k$ large enough, the order of the r.h.s.\ can be made arbitrarily small.

The estimates for $Y^\e$ follow from Lemma \ref{l:estPhiPsi}.
\end{proof}

\subsection{Accurate estimates for $X^\e$\label{s:X}}

Recall \eqref{e:U} and decompose the process $U^\e$ into the sum
\ba
\label{e:RAM}
U^\e_t
&= R_t + A_t^\e+M_t^\e,
\ea
with
\begin{align}
R_t&=b^0(y) t+\sum_{l=1}^m\sigma_{l}^0(y) W^l_t, \\
\label{e:AAA}
A_t^\e
&= \int_0^t   \Big(  \partial_x b^0(s)X^\e_{s\wedge\tau^\e}
+ \sum_{i=1}^n \partial_{y^i}b^0(s) \cdot (Y^{\e,i}_{s\wedge\tau^\e}-y^i)
+ \phi^{\e,0}(X^\e_{s\wedge\tau^\e},Y^\e_{s\wedge\tau^\e})\Big) \,\di s,\\
\label{e:MMM}
M^\e_t
&=\sum_{l=1}^m\int_0^t   \Big( \partial_x \sigma_{l}^0(s) X^\e_{s\wedge\tau^\e}
+ \sum_{i=1}^n \partial_{y^i}\sigma_{l}^0(s)  (Y^{\e,i}_{s\wedge\tau^\e}-y^i) + \phi_{l}^{\e,0}(X^\e_{s\wedge\tau^\e},Y^\e_{s\wedge\tau^\e})
\Big)\,\di W^l_s ,
\end{align}
where the terms $\partial_x b^0(s)$, $\partial_{y^i}b^0(s)$, etc.\
are bounded by Assumption \textbf{A}.
In the representation \eqref{e:RAM}, the process $R$ is a Brownian motion with drift, and $A^\e$ and $M^\e$ are ``small'' on $t\in[0,\tau^\e]$ as shown in the
next Lemma.

\begin{lem}
\label{l:AM}
For $\e>0$ small enough we have 
\ba
\E_{0,y} \sup_{t\leq \tau^\e}|A^\e_{\tau^\e}| \lesssim \e^3,\\
\E_{0,y}\sup_{t\leq \tau^\e}|M^\e_{\tau^\e}| \lesssim \e^2.
\ea
\end{lem}
\begin{proof}
To obtain these estimates, we use Remark \ref{r:small}, Lemmas \ref{l:tau} and \ref{l:V} as well as the Cauchy--Schwarz inequality:
\ba
\E_{0,y} \sup_{t\leq \tau^\e}|A^\e_{\tau^\e}| 
&\leq \E_{0,y}\int_0^{\tau^\e}  
\Big(  |\partial_x b^0(s)X^\e_s|
+ \sum_{i=1}^n |\partial_{y^i}b^0(s) \cdot (Y^{\e,i}_s-y^i)|
+ |\phi^{\e,0}(X^\e_s,Y^\e_s)|\Big) \,\di s\\
&\leq C\E_{0,y} \tau^\e\Big( \e + \sup_{t\leq \tau^\e}|Y^{\e}_{t}-y|\Big)\\
&\leq C\e \E_{0,y} \tau^\e  + \Big( \E_{0,y} |\tau^\e|^2\cdot \E_{0,y}\sup_{t\leq \tau^\e}|Y^{\e}_{t}-y|^2 \Big)^{1/2} \lesssim \e^3. 
\ea
To estimate the martingale $M^\e$ we apply the Doob inequality:
\ba
\E_{0,y}\sup_{t\leq \tau^\e}|M^\e_{\tau^\e}|^2
&\leq \sum_{l=1}^m \E_{0,y}\int_0^{\tau^\e} \Big( \partial_x
\sigma_{l}^0(s) X^\e_s
+ \sum_{i=1}^n
\partial_{y^i}\sigma_{l}^0(s)\cdot  (Y^{\e,i}_s-y^i) +
\phi_{l}^{\e,0}(X^\e_{s\wedge\tau^\e},Y^\e_s)
\Big)^2\,\di s\\
&\leq C\e^2 \E_{0,y}\tau^\e +  \Big( \E_{0,y} |\tau^\e|^2\cdot \E_{0,y}\sup_{t\leq \tau^\e}|Y^{\e}_{t}-y|^4 \Big)^{1/2} \lesssim \e^4. 
\ea
\end{proof}

We will prove that the exit probabilities of $X^\e$ from $(-a_-^\e,a_+^\e)$ coincide
with the exit probabilities of the Brownian motion with drift
$R$ from the modified interval $(-a_-^\e,a_+^\e(1-2\e\beta(y)))$ up to the terms of the order $\mathcal O(\e^{2-\delta})$
and calculate the moments of $X^\e_{\tau^\e}$ with accuracy of the order $\mathcal O(\e^{3-\delta})$. Note that the local time terms 
that disappeared upon transition to the process $U^\e$
contribute now to the asymmetry of the exit interval. 

\begin{lem}
\label{l:asympt}
For any $\delta\in(0,1)$ and for $\e>0$ small enough we have
\begin{align}
\label{e:p}
p_\pm^\e:=p_\pm^\e(y)&= \P_{0,y}\Big(X^\e_{\tau^\e}=\pm a^\e_\pm\Big)
= \frac{1}{2}\pm \frac12 \Big(\frac{b^0(y)d}{\Sigma^{00}(y)} + \beta(y)  -\frac{d'}{2d}\Big)\e+\mathcal O(\e^{2-\delta}),\\
% p_-^\e(y):= \P_{0,y}\Big(X^\e_{\tau^\e}=- d^\e_-\Big)
% &= \frac{1}{2}+ \frac12 \Big(-\frac{b^0(y)d}{\Sigma^{00}(y)} - \beta(y)  +\frac{d'}{2d}\Big)\e+\mathcal O(\e^{2-\delta}),\\
\label{e:EX}
\E_{0,y} X^\e_{\tau^\e} &= \Big( \beta(y)d  + \frac{b^0(y)d^2}{\Sigma^{00}(y)}   \Big) \e^2      +\mathcal{O}(\e^{3-\delta}),\\
\label{e:EX2}
\E_{0,y} \left(X^\e_{\tau^\e}\right)^2 &= d^2 \e^2      +\mathcal{O}(\e^{3-\delta}),\\
\label{e:tau}
\E_{0,y}\tau^\e&=  \frac{d^2}{\Sigma^{00}(y)}\e^2+  \mathcal{O}(\e^{3}).
\end{align}
\end{lem}

\begin{proof}
1. On the one hand we have
\ba
\E_{0,y} U^\e_{\tau^\e}
&=-a_-^\e p_-^\e  + a_+^\e\E_{0,y}[B^\e(Y^\e_{\tau^\e})\bI(X_{\tau^\e}^\e=a_+^\e) ].
\ea
The Taylor formula for the mapping $(\e,Y)\mapsto B^\e(Y)$ yields
\ba
B^\e(Y)
=1 -2\e \beta(y) + r(\e, Y-y ), 
\ea
where the remainder is estimated by 
\ba
|r(\e,Y-y)|\leq C(\e^2+ \e |Y-y|+ \e|Y-y|^2 )
\ea
for some $C>0$.
Hence we get
\ba
\E_{0,y}\Big[B^\e(Y^\e_{\tau^\e})\bI(X_{\tau^\e}^\e=a_+^\e) \Big]&= p_+^\e(1-2\e\beta(y))
+ \E_{0,y}\Big[ r(\e,Y^\e_{\tau^\e}-y) \bI(X_{\tau^\e}^\e=a_+^\e)  \Big]
\ea
and for $\delta\in (0,1)$ with the help of Lemma \ref{l:V} we estimate
\ba
\Big|\E_{0,y}\Big[ r(\e,Y^\e_{\tau^\e}-y) \bI(X_{\tau^\e}^\e=a_+^\e)  \Big] \Big|
&\leq \E_{0,y} |r(\e,Y^\e_{\tau^\e}-y)| 
=\E_{0,y}\Big[ |r(\e,Y^\e_{\tau^\e}-y)|  \bI(|Y^\e_{\tau^\e}-y|< \e^{1-\delta}) \Big]  \\
&+\E_{0,y}\Big[ |r(\e,Y^\e_{\tau^\e}-y)|  \bI(|Y^\e_{\tau^\e}-y|\geq \e^{1-\delta}) \Big]  
\\
&\lesssim  C\e ( \e^2 + \e^{2-\delta} + \e^{3-2\delta})\\
&+     \e\E_{0,y}\Big[ |Y^\e_{\tau^\e}-y| \bI(|Y^\e_{\tau^\e}-y|\geq \e^{1-\delta}) \Big]
+     \e\E_{0,y}\Big[ |Y^\e_{\tau^\e}-y|^2 \bI(|Y^\e_{\tau^\e}-y|\geq \e^{1-\delta}) \Big] \\
&\lesssim   \e^{3-\delta}  + \mathcal{O}(\e^\infty) \lesssim  \e^{3-\delta}.
\ea
Hence using \eqref{e:a} we get
\ba
\E_{0,y} U^\e_{\tau^\e}
&=-a_-^\e p_-^\e  + a_+^\e p_+^\e\Big(1-2\e\beta(y)\Big) + \mathcal O(\e^{3-\delta}) \\
&=-(1-p_+^\e) \Big(d\e -\frac{d'\e^2}{2}+\mathcal{O}(\e^3) \Big) + p_+^\e \Big(d\e +\frac{d'\e^2}{2}+\mathcal{O}(\e^3) \Big) \Big(1-2\e\beta(y) \Big)
+ \mathcal O(\e^{3-\delta})\\
&=2 p_+ d\e  -d\e +\frac{d'\e^2}{2}   + \mathcal O(\e^{3-\delta}).
\ea
2. On the other hand, taking into account \eqref{e:RAM} and \eqref{e:AAA} we estimate 
\ba
\label{e:EU}
\E_{0,y}U_{\tau^\e}^\e
&=\E_{0,y}  R_{\tau^\e} + \E_{0,y} A^\e_{\tau^\e}\\
&=\E_{0,y}  R_{\rho^\e} +\Big(\E_{0,y}  R_{\tau^\e} - \E_{0,y}  R_{\rho^\e}\Big)+ \mathcal O(\e^3),
\ea
where
\ba
\rho^\e:=\inf\Big\{t\geq 0\colon R_t\notin (-a_-^\e,a_+^\e(1-2\e\beta(y)))\Big\}.
\ea
To calculate $\E_{0,y} R_{\rho^\e}$,
use the explicit formula 3.0.4 (a) on p.\ 309 from Borodin and Salminen \cite{BorSal02}:
\ba
\P_{0,y}(R_{\rho^\e}=a_-^\e)&=\ex^{-\frac{b^0(y)a_-^\e}{\Sigma^{00}(y)}}
\frac{\sinh\Big(a_+^\e(1-2\e\beta(y))\frac{|b^0(y)|}{\Sigma^{00}(y)}\Big)}{\sinh\Big((a_+^\e(1-2\e\beta(y))+a_-^\e)\frac{|b^0(y)|}{\Sigma^{00}(y)}\Big)}\\
&=\frac{1}{2}+ \frac12 \Big(-\frac{b^0(y)d }{\Sigma^{00}(y)} - \beta(y)  +\frac{d' }{2d }\Big)\e+\mathcal O(\e^2)
\ea
that gives us the asymptotics
\ba
\label{e:ER}
\E_{0,y} R_{\rho^\e}&=-a_-^\e \P_{0,y}(R_{\rho^\e}=a_-^\e)  + a_+^\e(1-2\e\beta(y))\P_{0,y}(R_{\rho^\e}=a_+^\e(1-2\e\beta(y)))\\
&=\frac{b^0(y)d^2}{\Sigma^{00}(y)}\e^2 + \mathcal O(\e^3).
\ea
The mean value $\E_{0,y} \rho^\e$ is obtained with the help of the formula 3.0.1 on p.\ 309 from Borodin and Salminen \cite{BorSal02}:
\ba
\label{e:Erho}
\E_{0,y} \rho^\e&=\frac{d^2}{\Sigma^{00}(y)}\e^2+ \mathcal O(\e^3).
\ea
Note that
\ba
|\E_{0,y} R_{\tau^\e} - \E_{0,y} R_{\rho^\e}   |\leq \|b\|_\infty\E_{0,y}|\tau^\e-\rho^\e|.
\ea
Hence by Lemma \ref{l:taurho} (below) we get the estimate
\ba
|\E_{0,y} R_{\tau^\e} - \E_{0,y} R_{\rho^\e}   | \lesssim \e^3.
\ea
Combining \eqref{e:EU} and \eqref{e:ER} we get \eqref{e:p} and hence immediately  \eqref{e:EX} and \eqref{e:EX2}.
The asymptotics \eqref{e:tau} follows from \eqref{e:Erho} and Lemma \ref{l:taurho}.
\end{proof}

\begin{lem}
\label{l:taurho}
For $\e>0$ small enough we have
\ba
\label{eq:ineq_diff_exits}
\E_{0,y}|\tau^\e-\rho^{\e}|\lesssim \e^3.
\ea
\end{lem}

\begin{proof}
1.
Let $\widetilde a_+^\e=a_+^\e(1-2\e\beta(y))$ and
\ba
h(u)=(\widetilde a_+^\e-u)(u+a_-^\e),
\ea
so that $h'(u)=-2 u + \widetilde a_+^\e-a_-^\e$, $h''(u)=-2$.
The It\^o formula on the event $\{\rho^\e\leq\tau^\e\}$
yields
\ba
\E_{0,y}\Big[ h(U_{\tau^\e}^\e)\Big|\rF_{\rho^\e}\Big]&=
h(U^\e_{\rho^\e})
+\E_{0,y}\Big[\int_{\rho^\e}^{\tau^\e}(-2 U_s^\e + a_+^\e-a_-
^\e)  (b^0(X^\e_s,Y^\e_s)+\phi^{\e,0}(X^\e_s,Y^\e_s))\,\di
s\Big|\rF_{\rho^\e}\Big]\\
&-  \sum_{l=1}^m \E_{0,y}\Big[\int_{\rho^\e}^{\tau^\e} \Big(
\sigma_{l}^0(X^\e_s,Y^\e_s)
+ \phi_{l}^{\e,0}(X^\e_s,Y^\e_s)\Big)^2\,\di s\Big|\rF_{\rho^\e}\Big]\\
&\leq h(U^\e_{\rho^\e})
+C_1\e \E_{0,y}\Big[\tau^\e
-\rho^\e\Big|\rF_{\rho^\e}\Big]-  \delta_1\E_{0,y}\Big[\tau^\e-
\rho^\e\Big|\rF_{\rho^\e}\Big]
\ea
where we used that $\sum_{l=1}^m \sigma_{l}^0(x,y)^2\geq \delta_1>0$,
$|b^0(x,y)+\phi^{\e,0}(x,y)|\leq \delta_2$,
and $|U^\e_t|\leq C\e$, for some $C>0$ fixed and $\e>0$ sufficiently
small. We get
\ba
\frac{\delta_1}{2} \E \Big[ \tau^\e-\rho^\e\Big|\rF_{\rho^\e}\Big]
\leq   h(U^\e_{\rho^\e})
-\E\Big[ h(U_{\tau^\e})\Big|\rF_{\rho^\e}\Big].
\ea
We also have
\ba
h(U^\e_{\rho^\e})
&=\Big(\widetilde a_+^\e-R_{\rho^\e}-
(A^\e_{\rho^\e}+M^\e_{\rho^\e})\Big)\Big( R_{\rho^\e}+a_-
^\e+(A^\e_{\rho^\e}+M^\e_{\rho^\e})\Big)\\
&\leq (\widetilde a_+^\e-R_{\rho^\e})(R_{\rho^\e}+a_-^\e) +
(A^\e_{\rho^\e}+M^\e_{\rho^\e})(\widetilde a_+^\e-2R_{\rho^\e}-a_-
^\e)\\
&=(A^\e_{\rho^\e}+M^\e_{\rho^\e})(\widetilde a_+^\e-2R_{\rho^\e}-a_-
^\e)\\
\ea
The second factor in the latter formula is of the order 
$\mathcal{O}(\e)$, and by Lemma \ref{l:AM}
we have
\ba
\E_{0,y} |h(U^\e_{\rho^\e})| &  \lesssim \e^3.
\ea
Furthermore,
\ba
\E_{0,y}| h_\e(U_{\tau^\e})|&=\E_{0,y} \Big|(\widetilde a_+^\e-
U^\e_{\tau^\e})(U^\e_{\tau^\e}+a_-^\e)\Big|\\
&= \E_{0,y} \Big|(\widetilde a_+^\e-U^\e_{\tau^\e})(U^\e_{\tau^\e}+a_-
^\e)\Big|  \bI(X^\e_{\tau^\e}=a_+^\e)\\
&\leq C\e \E_{0,y} \Big|\widetilde a_+^\e-
U^\e_{\tau^\e}\Big|  \bI(X^\e_{\tau^\e}=a_+^\e)\\
&\leq C\e^2 \E_{0,y} \Big| 1-2\e\beta(y)
-B^\e(Y^\e_{\tau^\e})\Big|  \bI(X^\e_{\tau^\e}=a_+^\e)\\
&\leq C\e^2 \E_{0,y} |r(\e,Y^\e_{\tau^\e}-
y)|\bI(X^\e_{\tau^\e}=a_+^\e)\\
&\lesssim \e^{4-\delta}.
\ea
2.
On the event $\{\tau^\e\leq \rho^\e\}$ we have analogously:
\ba
R_{\tau^\e}=U^\e_{\tau^\e} - M_{\tau^\e}^\e- A_{\tau^\e}^\e,\\
\ea
and
\ba
0&=\E\Big[ h(R_{\rho^\e})\Big|\rF_{\tau^\e}\Big]
= h(R_{\tau^\e})
+b(y)\E\Big[\int^{\rho^\e}_{\tau^\e}(-2 R_s^\e +a_+^\e-a_-^\e)\,\di
s\Big|\rF_{\tau^\e}\Big]
-  \Sigma^{00}(y)\E\Big[\int^{\rho^\e}_{\tau^\e}\,\di
s\Big|\rF_{\tau^\e}\Big]\\
\ea
Hence,
\ba
\frac{\Sigma^{00}(y)}{2} \E\Big[\rho^\e -
\tau^\e\Big|\rF_{\tau^\e}\Big]&\leq  h(R_{\tau^\e})
% &=(\widetilde d_+^\e- U_{\tau^\e}^\e
+(A_{\tau^\e}^\e+M_{\tau^\e}^\e)  )( U^\e_{\tau^\e}+a_-^\e-
(A_{\tau^\e}^\e+M_{\tau^\e}^\e))\\
\leq h(U_{\tau^\e}^\e) +(A_{\tau^\e}^\e+M_{\tau^\e}^\e)(
2U^\e_{\tau^\e}+a_-^\e - \widetilde a_+^\e)
\ea
and the estimates on $ h(U_{\tau^\e}^\e)$ and $A_{\tau^\e}^\e$, and
$M_{\tau^\e}^\e$ apply as in the previous step.

Combining the estimates in steps 1 and 2 yields
\eqref{eq:ineq_diff_exits}.
\end{proof}

\subsection{Accurate estimates containing $Y^\e$\label{s:Y}}

In this section we sketch the derivation of the expectations 
$\E_{0,y}(Y^{i,\e}_{\tau^\e}-y_i)(Y^{j,\e}_{\tau^\e}-y_j)$,
$\E_{0,y}(Y^{i,\e}_{\tau^\e}-y_i)X_{\tau^\e}^\e$, and
$\E_{0,y}(Y^{i,\e}_{\tau^\e}-y^i)$.

\noindent 
1. For $1\leq i,j\leq n$ we have:
\ba
\E_{0,y}(Y^{i,\e}_{\tau^\e}-y_i)(Y^{j,\e}_{\tau^\e}-y_j)
&=\E_{0,y}(V^{i,\e}_{\tau^\e}-y_i+\theta^i(Y^\e_\tau)X^\e_{\tau^\e})(V^{j,\e}_{\tau^\e}-y_j +\theta^j(Y^\e_{\tau^\e})X^\e_{\tau^\e})\\
&=\E_{0,y}(V^{i,\e}_{\tau^\e}-y_i)(V^{j,\e}_{\tau^\e}-y_j)\\
&+ \E_{0,y}(V^{i,\e}_{\tau^\e}-y_i)\theta^j(Y^\e_{\tau^\e})X^\e_{\tau^\e}\\
&+\E_{0,y} \theta^i(Y^\e_{\tau^\e})X^\e_{\tau^\e}(V^{j,\e}_{\tau^\e}-y_j)\\
&+\E_{0,y}\theta^i(Y^\e_{\tau^\e})\theta^j(Y^\e_{\tau^\e})|X^\e_{\tau^\e}|^2.
\ea
Let us apply the It\^o formula, Lemma \ref{l:tau}, and Lemma \ref{l:V}. We write only the essential terms; the $\mathcal O(\e^3)$ estimate follows from the bounds of Remark \ref{r:small}): 
\ba
\E_{0,y}  (V^{i,\e}_{\tau^\e}-y^i)(V^{j,\e}_{\tau^\e}-y^j)
&=\sum_{l=1}^m \E_{0,y}
\int_0^{\tau^\e} \Big(\sigma^i_l(X_s,Y_s) - \theta^i(Y_s)\sigma_{l}^0(X_s,Y_s)\Big)\Big(\sigma_{l}^j(X_s,Y_s) - \theta^j(Y_s)\sigma_{l}^0(X_s,Y_s)\Big)\,\di s
+\mathcal O(\e^3)\\
&=\sum_{l=1}^m \E_{0,y}
\int_0^{\tau^\e} \Big(\sigma_{l}^i(0,y) - \theta^i(y)\sigma_{l}^0(0,y)\Big)\Big(\sigma_{l}^j(0,y) - \theta^j(y)\sigma_{l}^0(0,y)\Big)\,\di s
+\mathcal O(\e^3)\\
&=\Big(  \Sigma^{ij}(y) -\theta^i(y) \Sigma^{0j}(y)   -\theta^j(y) \Sigma^{0i}(y) + \theta^i(y)\theta^j(y) \Sigma^{00}(y)   \Big)
\E_{0,y}\tau^\e +\mathcal O(\e^3).
\ea
Applying the Taylor formula to $\theta^i$, $\theta^j$ and using \eqref{e:EX2} and \eqref{e:tau} we get
\ba
\E_{0,y}\theta^i(Y^\e_{\tau^\e})\theta^j(Y^\e_{\tau^\e})|X^\e_{\tau^\e}|^2
& =\theta^i(y)\theta^j(y)\E_{0,y}|X^\e_{\tau^\e}|^2 
+ \mathcal O(\E_{0,y}|Y^\e_{\tau^\e}-y|^2 |X^\e_{\tau^\e}|^2) \\
&=\theta^i(y)\theta^j(y)d^2\e^2 + \mathcal O(\e^{3-\delta})+\mathcal O(\e^4) \\
&=\theta^i(y)\theta^j(y) \Sigma^{00}(y)  \E_{0,y}\tau^\e + \mathcal O(\e^{3-\delta}); \\
\ea
Analogously,
\ba
\E_{0,y}(V^{\e,i}_{\tau^\e}-y_i)\theta^j(Y^\e_{\tau^\e})X^\e_{\tau^\e}
&=\E_{0,y}(Y^{\e,i}_{\tau^\e}-y_i-\theta^i(Y_{\tau^\e}^\e)X_{\tau^\e}^\e)\theta^j(Y^\e_{\tau^\e})X^\e_{\tau^\e}\\
&=\theta^j(y)\E_{0,y}(Y^{i,\e}_{\tau^\e}-y_i)X^\e_{\tau^\e} - \theta^i(y)\theta^j(y) \E_{0,y} |X^\e_{\tau^\e}|^2 +\mathcal O(\e^3)\\
&=\theta^j(y)\Sigma^{0i}(y) \E_{0,y}\tau^\e - \theta^i(y)\theta^j(y) \Sigma^{00}(y)  \E_{0,y}\tau^\e
+ \mathcal O(\e^3).
\ea
Hence
 \ba
 \label{e:EY2}
\E_{0,y}(Y^{i,\e}_{\tau^\e}-y_i)(Y^{j,\e}_{\tau^\e}-y_j)=\Sigma^{ij}(y)\E_{0,y}\tau^\e +\mathcal O(\e^{3-\delta}).
\ea

\noindent 
2. Let $1\leq i\leq n$. 
Recall Lemma \ref{l:estPhiPsi}, i.e., $\Phi(u,v)=u+h_1^\e(u,v)$, $h_1^\e(u,v)=\mathcal O(\e^2)$ uniformly in $(u,v)\in (-a^\e_-,a^\e_+)\times \bR^n$. Then
\ba
\E_{0,y}(Y^{i,\e}_{\tau^\e}-y_i)X_{\tau^\e}^\e&= \E_{0,y}  (\Psi^\e(U^{i,\e}_{\tau^\e},V^{i,\e}_{\tau^\e})-y_i)\Phi^{\e}(U_{\tau^\e}^\e,V_{\tau^\e}^\e)\\
&=\E_{0,y}  (V^{i,\e}_{\tau^\e}-y^i +\theta^i(V^\e_{\tau^\e})U^\e_{\tau^\e} +h^\e_2)(U_{\tau^\e} + h_1^\e)\\
&=\E_{0,y}  (V^{i,\e}_{\tau^\e}-y^i)U_{\tau^\e}^\e
+\E_{0,y}  \theta^i(V^\e_{\tau^\e})|U^\e_{\tau^\e}|^2+ \mathcal O(\e^3)\\
&=\E_{0,y}\Big[ \sum_{l=1}^m \int_0^{\tau^\e} \Big[  \sigma_{l}^i(X_s^\e,Y_s^\e) - \theta^i(Y_s^\e)\sigma_{l}^0(X_s^\e,Y_s^\e) \Big]\,\di W^l_s\cdot
\sum_{l=1}^m \int_0^{\tau^\e}   \sigma_{l}^0(X_s^\e,Y_s^\e) \,\di W^l_s \Big]\\
&+\theta^i(y)\sum_{l=1}^m  \E_{0,y} \int_0^{\tau^\e}  \sigma_{l}^0(X_s^\e,Y_s^\e)^2\,\di s+ \mathcal O(\e^3)\\
% & \approx\sum_{l=0}^n  \E \int_0^\tau \Big[  \sigma_{il}(X_s,Y_s) - \theta^i(Y_s)\sigma_{0l}(X_s,Y_s) \Big] \sigma_{0l}(X_s,Y_s) \,\di s
&=\Sigma^{0i}(y) \E_{0,y}\tau^\e+ \mathcal O(\e^3).
\ea

\noindent
3. With the help of \eqref{e:Psi} we obtain
\ba
\label{e:EY}
\E_{0,y}&(Y^{i,\e}_{\tau^\e}-y^i)= \E_{0,y}  (\Psi^{i,\e}_{\tau^\e}(U^\e_{\tau^\e},V^\e_{\tau^\e})-y^i)\\
&=\E_{0,y}  (V^{i,\e}_{\tau^\e}-y^i +\theta^i(V^\e_{\tau^\e})U^\e_{\tau^\e} +h_2^\e  ) \\
&=\E_{0,y}  \Big[ V^{i,\e}_{\tau^\e}-y^i
-\theta^i(y)a_-^\e  \bI(X^\e_{\tau^\e}=-a_-^\e)
+\theta^i(y)a_+^\e(1+2\e\beta(y)) \bI(X^\e_{\tau^\e}=a_+^\e)
+ \sum_{j=1}^n \theta^i_{y^j}(y)(V^{\e,i}_{\tau^\e}-y)U_{\tau^\e}^\e\Big]\\
&+ \mathcal{O} \Big(\E_{0,y}\Big[|Y^\e_{\tau^\e}-y|^2 U_{\tau^\e}^\e +h_2^\e  \Big]\Big) \\
&=\E_{0,y} \int_0^{\tau^\e} \Big( b^i(y) -\theta^i(y) b^0(y) -\sum_{j=1}^n \theta^i_{y^j}(y)\Sigma^{0j}(y) \Big)\,\di s
+\theta^i(y)\Big(\beta(y)d +\frac{b^0(y)d^2}{\Sigma^{00}(x,y)}\Big)\e^2+\mathcal O(\e^3)\\
 &= \Big( b^i(y) + \frac{\theta^i(y)\beta(y) \Sigma^{00}(y)}{d}  \Big)\E_{0,y}\tau^\e +\mathcal O(\e^3).
\ea

\section{Proof of Theorem \ref{t:hom}\label{s:conv}}

Let $(X^\e,Y^\e)$, $\e\in(0,\e_0]$, be weak solutions of \eqref{e:XY} 
 and $(X,Y)$ be a weak solution of the limit SDE \eqref{e:XYlim}. Recall that all these processes are strong Markov.

The proof of the weak convergence $(X^\e,Y^\e) \Rightarrow (X,Y)$ in $C([0,\infty),\bR^n)$ consists, as usual,
in the proof of the weak relative compactness of the family $(X^\e,Y^\e)$ and the proof of the
convergence of finite dimensional distributions.

Denote by $S^\e(a_j^\e)=(a_{j-1}^\e,a_{j+1}^\e)$ the stripe around the $j$-th membrane.
Let 
\ba
\label{e:tau0}
\tau_0^\e&=\inf\Big\{t\geq 0\colon X^\e_t \in \{a_j^\e\}_{j\in\mathbb Z}\Big\},\\
% \tau_1^\e&=\inf\{t> \tau_0^\e\colon X^\e_t\notin S^\e( X^\e_{\tau_1^\e})\},\\
\tau_{k+1}^\e&=\inf\{t> \tau_k^\e\colon X^\e_t\notin S^\e(X^\e_{\tau_k^\e})\},\quad k\geq 0,
\ea
be the hitting times of the neighbouring membranes.
Let $\nu^\e=(\nu^\e_t)$ be a counting process of visits to $\{a_k^\e\}$, i.e.,
\ba
\nu^\e_t=k\quad \Leftrightarrow\quad \tau_0^\e<\cdots<\tau_k^\e\leq t,\ \tau_{k+1}^\e>t.
\ea

\begin{lem}
\label{l:nu}
For any $T>0$ there is $K>0$ such that for each $\lambda\in(0,T]$ 
\ba
\lim_{\e\to 0}\sup_{x\in \{a_j^\e\},y\in\bR^n} \P_{x,y} ( \nu_\lambda^\e> K\lambda/\e^2)=0.
\ea

\end{lem}
\begin{proof}
Recall that according to
Lemma \ref{l:tau}
there are $C_1,C_2>0$ such that $\E_{x,y}\tau_1^\e\geq C_1\e^2$ and $\E_{x,y}|\tau_1^\e|^2\leq C_2\e^4$ uniformly for all $x\in\{a_k^\e\}$, $y\in\bR^n$ and $\e\in(0,\e_0]$ . 

Clearly, $\{\nu^\e_\lambda\geq K\lambda/\e^2\}=\{\tau_{[K\lambda/\e^2]}^\e\leq \lambda\}$.
The random sequence
\ba
\Delta^\e_k:=\tau^\e_{k}-\tau^\e_{k-1}-\E \Big[\tau^\e_{k}-\tau^\e_{k-1}\Big|\rF^\e_{\tau_{k-1}^\e}\Big]
\ea
is a martingale difference and 
\ba
\E_{x,y}|\Delta_k^\e|^2\leq 4\E_{x,y}|\tau^\e_{k}-\tau^\e_{k-1}|^2\leq 4C_2\e^4.
\ea
By the strong Markov property for any $\lambda\in[0,T]$ we have:%, $\Delta_k^\e$ are independent. Hence
\ba
\P_{x,y} ( \nu_\lambda^\e\geq K\lambda/\e^2) 
&= \P_{x,y} \Big(\tau^\e_{[K\lambda/\e^2]}\leq \lambda  \Big)
\\
&= \P_{x,y} \Big(
\sum_{k=1}^{[K\lambda/\e^2]}\Delta^\e_k
\leq \lambda -  \sum_{k=1}^{[K\lambda/\e^2]} \E [\tau^\e_{k}-\tau^\e_{k-1}|\rF^\e_{\tau_{k-1}^\e}]   \Big)
\\
&\leq  \P_{x,y} \Big(
\sum_{k=1}^{[K\lambda/\e^2]}\Delta^\e_k  \leq \lambda -  C_1 [K\lambda/\e^2]\e^2 \Big)\\
&=\P_{x,y} \Big(
% -A\lambda \leq 
\frac{1}{K}\sum_{k=1}^{[K\lambda/\e^2]}\Delta^\e_k  \leq \frac{\lambda}{K} -C_1 \frac{[K\lambda/\e^2]\e^2}{K} \Big)\\
&\leq 
\P_{x,y} \Big(
\Big| \frac{1}{K}\sum_{k=1}^{[K\lambda/\e^2]}\Delta^\e_k \Big| \geq \frac{C_1\lambda}{2}\Big)\\
% &\leq \P_{x,y} \Big(\frac{1}{N}\Big|\sum_{k=1}^{[N\lambda/\e^2]}\Delta^\e_k \Big| \leq A (\sqrt\lambda\vee\lambda)\Big) 
\ea
for $K\geq 2/C_1$. 
By Markov's inequality
\ba
\P_{x,y} \Big(\Big|\frac{1}{K}\sum_{k=1}^{[K\lambda/\e^2]}\Delta^\e_k\Big|> \frac{C_1\lambda}{2}\Big)
\leq \frac{4}{ C_1^2 K^2 \lambda^2 }\E_{x,y} \Big|\sum_{k=1}^{[K\lambda/\e^2]}\Delta^\e_k\Big|^2
\leq  \frac{4C_2\e^2}{ C_1^2K \lambda } 
\ea
which vanishes in the limit $\e \to 0$.
\end{proof}

\subsection{Weak relative compactness\label{s:wrc}}

The weak relative compactness of the family $(X^\e,Y^\e)$ in $C(\bR_+,\bR^{n+1})$ will follow from the its weak relative compactness
in the Skorokhod space $D(\bR_+,\bR^{n+1})$ guaranteed by the 
compact containment condition and Aldous' criterion 
(e.g., see conditions 3.21i and 4.4, and Theorem 4.5 in Chapter VI in Jacod and Shiryaev \cite{JacodS-03}), 
and the continuity of the paths of $(X^\e,Y^\e)$.

\begin{lem}[compact containment]\label{lem:compactCOnt}
For any $T>0$ and $\delta>0$ and there is $R=R_{\delta,T}>0$ such that
\ba
\label{e:cc}
\liminf_{\e\to 0}\P_{x,y}\Big( |X^\e_t|\leq R, |Y^\e_t|\leq R,\, t\in[0,T]\Big)\geq 1-\delta
\ea
uniformly over $x\in\bR$ and $y\in \bR^n$ belonging to compacts.
\end{lem}

\begin{proof}
1. Recall that $\tau_0^\e=\inf\{t\geq 0\colon X_t\in\{a_k^\e\}\}$ is the first hitting time of a membrane. 
Since the process $(X^\e,Y^\e)$ is a regular diffusion on the interval $[0,\tau^\e_0]$ it is easy to see that
\ba
\label{e:tau00}
\sup_{(x,y)\in\bR^{n+1}}\E \tau_0^\e\lesssim \e^2,\quad 
\sup_{(x,y)\in\bR^{n+1}}\E |X_{\tau_0^\e}^\e-x|\lesssim \e,\quad 
\sup_{(x,y)\in\bR^{n+1}}\E |Y_{\tau_0^\e}^\e-y|\lesssim \e.
\ea
Hence, in order to prove \eqref{e:cc} it suffices to consider starting points $(x,y)\in \{a_k^\e\}\times \bR^n$.

\noindent
2. Let $\delta\in(0,1)$.
By Lemma \ref{l:nu}, there is $K>0$ and $\e_0(\delta)>0$ such that for all $\e\leq \e_0(\delta)$
\ba
\sup_{x\in \{a_k^\e\},y\in\bR^n} \P_{x,y}(\e^2 \nu_T^\e \geq KT)\leq \frac{\delta}{6} .
\ea
3. With abuse of notation, denote now by $(\Delta_{k}^\e)_{k\geq 1}$ the martingale difference 
\ba
\label{e:DeltaX}
\Delta_{k}^\e:=X^\e_{\tau^\e_{k}}-X^\e_{\tau^\e_{k-1}}
-\E_{X^{\e}_{\tau_{k-1}^\e},Y^{\e}_{\tau_{k-1}^\e}}(X^\e_{\tau^\e_{k}}-X^\e_{\tau^\e_{k-1}}  ),\quad  k\geq 1.
\ea
Due to the strong Markov property
and 
\eqref{e:EX} we have
\ba
\Big|\E_{X^{\e}_{\tau_{k-1}^\e},Y^{\e}_{\tau_{k-1}^\e}}(X^\e_{\tau^\e_k}-X^\e_{\tau^\e_{k-1}}  )\Big| \leq C_1\e^2 \text{ a.s.},
\ea
so that due to \eqref{e:EX2} we get
\ba
\E_{x,y}|\Delta_{k}^\e|^2\leq C_2\e^2
\ea
for some constants $C_1,C_2>0$.
By Doob's inequality for $\e>0$ small enough and $R>0$ large enough we get
\ba
\label{e:X1}
\P_{x,y}\Big(\sup_{t\in [0,T]}|X_t^\e|>R, \nu_T^\e\leq KT/ \e^2 \Big)
&\leq
\P_{x,y}\Big(\max_{k\leq  KT\e^{-2}}\Big|\sum_{j=1}^k (X^\e_{\tau^\e_k}- X^\e_{\tau^\e_{k-1}})\Big|>R - 2\|d\|_\infty\e -|x|\Big)\\
&\leq \P_{x,y}\Big(\max_{k\leq  KT\e^{-2}}\Big|\sum_{j=1}^k \Delta^\e_k \Big|>R - 2\|d\|_\infty\e -C_1KT-|x|\Big)\\
&\leq \frac{ C_2 KT }{(R - 2\|d\|_\infty\e -C_1KT -|x| )^2}
\leq \frac{\delta}{6}.
\ea
4.
To estimate $Y^\e$ we recall that in each strip 
$S^\e(a_{\tau^\e_k}^\e)=S^\e(X^\e_{\tau^\e_k})$, i.e., on each time interval $[\tau_k^\e,\tau_{k+1}^\e]$, we can perform the transformation
\ba
\label{e:mm}
V_t^\e&=Y_t^\e-\theta(Y_t^\e)(X_t^\e-X^\e_{\tau^\e_k}).
\ea
The process 
\ba
t\mapsto \theta(Y_t^\e)(X_t^\e-X^\e_{\tau^\e_k}),\quad [\tau_k^\e,\tau_{k+1}^\e],\ k\geq 0,
\ea
is uniformly bounded by $\|d\|_\infty\|\theta\|_\infty\e$. 

Note that $V_{\tau_k^\e}^\e=Y_{\tau_k^\e}^\e$, and on $t\in[\tau_k^\e,\tau_{k+1}^\e]$
the process $V^\e$ has the representation 
\ba
\label{e:mmm}
V^{\e, i}_t=Y^{\e,i}_{\tau_k^\e}
&+\int_{\tau_k^\e}^t \Big( b^i -\theta^i b^0 - \sum_{j=1}^n  \theta^i_{y^j}\Sigma^{0j} + \psi^{i}(\cdot,\cdot;X^\e_{\tau^\e_k})\Big)(X_s^\e,Y_s^\e) \,\di s\\
&+\sum_{l=1}^m \int_{\tau_k^\e}^t \Big(  \sigma_{l}^i - \theta^i \sigma_{l}^0 +\psi_{l}^{i}(\cdot,\cdot; X^\e_{\tau^\e_k})  \Big)(X_s^\e,Y_s^\e)\,\di W^l_s
\ea
with 
\ba
\psi^{i}(x,y;a)
&=-  \sum_{j=1}^n  \theta_{y^j}^i(y) (x-a)   b^j(x,y)
-\frac12 \sum_{j,k=1}^n \theta^i_{y^jy^k}(y) (x-a)\Sigma^{jk}(x,y),\\
\psi_{l}^{i}(x,y;a)&=-\sum_{j=1}^n \theta^i_{y^j}(y)  (x-a) \sigma_{l}^j(x,y),\quad i=1,\dots,n,\ l=1,\dots,m.
\ea
Thus combining \eqref{e:mm} and \eqref{e:mmm} we have on $[\tau_k^\e,\tau_{k+1}^\e]$
\ba
Y_t^\e=y+\sum_{j=1}^k \theta(Y_{\tau_j^\e}^\e) (X_{\tau_j^\e}^\e-X^\e_{\tau^\e_{j-1}})+ I^\e_t+ \theta(Y^\e_t) (X_t^\e-X^\e_{\tau^\e_k})
\ea
where the process $I^\e$ is a 
diffusion
\ba
I_t^\e=\int_0^t \Gamma(X_s^\e,Y_s^\e,\omega)\, \di s+ \int_0^t \Lambda(X_s^\e,Y_s^\e,\omega)\,\di W(s)
\ea
with uniformly bounded vector- and matrix-valued functions $\Gamma$ and $\Lambda$ that obviously satisfies
\ba
\sup_{\e\in(0,\e_0]}\P_{x,y}\Big(\sup_{t\in[0,T]}|I^\e_t+y|>R\Big)\leq \frac{\delta}{6},\quad R\to\infty.
\ea

\noindent
5. In order to prove the compact containment condition for $Y^\e$ it is sufficient to estimate the sum
$\sum_{j=1}^k \theta(Y_{\tau_j^\e}^\e)(X_{\tau_j^\e}^\e-X^\e_{\tau^\e_{j-1}})$. First, the following elementary estimates hold true:
\ba
\Big|\sum_{j=1}^k &  \theta(Y_{\tau_j^\e}^\e) (X_{\tau_j^\e}^\e-X^\e_{\tau^\e_{j-1}})\Big|\\
&\leq \Big|\sum_{j=1}^k \theta(Y_{\tau_{j-1}^\e}^\e) (X_{\tau_j^\e}^\e-X^\e_{\tau^\e_{j-1}})\Big|
+ \Big|\sum_{j=1}^k (\theta(Y_{\tau_j^\e}^\e)-\theta(Y_{\tau_{j-1}^\e}^\e)) (X_{\tau_j^\e}^\e-X^\e_{\tau^\e_{j-1}})\Big|\\
&\leq \Big|\sum_{j=1}^k \theta(Y_{\tau_{j-1}^\e}^\e)
  \Big((X_{\tau_j^\e}^\e-X^\e_{\tau^\e_{j-1}})-\E\Big[ X_{\tau_j^\e}^\e-X^\e_{\tau^\e_{j-1}}\Big|\rF_{\tau^\e_{j-1}} \Big]   \Big)\Big|\\
&+ \Big|\sum_{j=1}^k \theta(Y_{\tau_{j-1}}^\e)\E\Big[ X_{\tau_j^\e}^\e-X^\e_{\tau^\e_{j-1}}\Big|\rF_{\tau^\e_{j-1}} \Big]   
\Big|
+ 2\sum_{j=1}^k (X_{\tau_j^\e}^\e-X^\e_{\tau^\e_{j-1}})^2 +2\|D\theta\|_\infty^2\sum_{j=1}^k |Y_{\tau_j^\e}^\e-Y_{\tau_{j-1}^\e}^\e|^2
\ea
Since $(X_{\tau_j^\e}^\e-X^\e_{\tau^\e_{j-1}})^2\leq C\e^2$ for all $j\geq 1$, for the third sum we have 
\ba
\P_{x,y}\Big(\max_{k\leq KT/\e^2} \sum_{j=1}^k (X_{\tau_j^\e}^\e-X^\e_{\tau^\e_{j-1}})^2>R\Big)\leq \P_{x,y}(KTC>R)=0
\ea
for $R$ large. For the fourth summand we use the fact that  $\E_{x,y}|Y_{\tau_j^\e}^\e-Y^\e_{\tau^\e_{j-1}}|^2\leq C\e^2$ 
(see \eqref{e:EY2}) and Markov's inequality to get
\ba
\P_{x,y}\Big(\max_{k\leq KT/\e^2} \sum_{j=1}^k |Y_{\tau_j^\e}^\e-Y^\e_{\tau^\e_{j-1}}|^2>R\Big)\leq 
\frac{KTC}{R}\leq \frac{\delta}{6},\quad R\to\infty.
\ea
Due to \eqref{e:EX} we have
\ba
\Big|\sum_{j=1}^k \theta(Y_{\tau_{j-1}}^\e)\E\Big[ X_{\tau_j^\e}^\e-X^\e_{\tau^\e_{j-1}}\Big|\rF_{\tau^\e_{j-1}} \Big]\Big|
=\Big| \sum_{j=1}^k m(X_{\tau_{j-1}}^\e , Y_{\tau_{j-1}}^\e) \e^2 + \mathcal O(\e^{3-\delta}) \Big|\leq Ck\e^2 
\ea
where
\ba
m(x,y)=\theta(y)d(x)\Big(\beta(x,y)    + \frac{b^0(x,y)d(x)   }{\Sigma^{00}(x,y)} \Big)
\ea
is a bounded function, and the $\mathcal{O}$ term is uniformly bounded by a.s. Hence we get
\ba
\P_{x,y}\Big(\max_{k\leq KT/\e^2} \Big|\sum_{j=1}^k \theta(Y_{\tau_{j-1}}^\e)\E\Big[ X_{\tau_j^\e}^\e-X^\e_{\tau^\e_{j-1}}\Big|\rF_{\tau^\e_{j-1}} \Big]\Big|       >R\Big)\leq \P(KTC>R  )=0
,\quad R\to\infty.
\ea
Again, with abuse of notation, the sequence
\ba
\Delta_k^\e=\theta(Y_{\tau_{k-1}}^\e)
\Big( (X_{\tau_k^\e}^\e-X^\e_{\tau^\e_{k-1}}) -\E\Big[ X_{\tau_k^\e}^\e-X^\e_{\tau^\e_{k-1}}\Big|\rF_{\tau^\e_{k-1}} \Big]\Big)
\ea
is a martingale difference, so that by Doob's ineguality we get
\ba
\P_{x,y}\Big(\max_{k\leq KT/\e^2}\Big|\sum_{j=1}^k \Delta_k^\e\Big|>R\Big)\leq \frac{C}{R^2}\sum_{k\leq KT/\e^2}\E_{x,y} |\Delta_k^\e|^2
\leq \frac{C_1KT}{R^2}\leq \frac{\delta}{6},\quad R\to\infty.
\ea 
\end{proof}

\begin{lem}[Aldous' criterion]
For any $\eta>0$
\ba
\lim_{\lambda\to 0} \limsup_{\e\to 0}\sup_{\sigma\leq \tau\leq \sigma+\lambda}\P_{x,y}\Big(|X_\tau^\e-X^\e_\sigma|+ |Y_\tau^\e-Y^\e_\sigma|>\eta\Big)=0,
\ea
where $\sigma$ and $\tau$ are $\mathbb F^\e$-stopping times, and the limit holds uniformly over $x\in \bR$, $y\in \bR^n$.
\end{lem}
\begin{proof}
Let $\eta>0$. With the help of the strong Markov property we get
\ba
\sup_{\sigma\leq \tau\leq \sigma+\lambda}\P_{x,y}\Big(|X_\tau^\e-X^\e_\sigma|+ |Y_\tau^\e-Y^\e_\sigma|>\eta\Big)
&=\sup_{0\leq \sigma\leq \lambda}\sup_{x,y}\P_{x,y}\Big(|X^\e_\sigma-x|+ |Y^\e_\sigma-y|>\eta\Big)\\
&\leq \sup_{x,y}\P_{x,y}\Big(\sup_{t\in[0,\lambda] }|X^\e_t-x|>\eta/2, \nu_\lambda^\e\leq K\lambda/\e^2\Big)\\
&+\sup_{x,y}\P_{x,y}\Big(\sup_{t\in[0,\lambda] }|Y^\e_t-y|>\eta/2, \nu_\lambda^\e\leq  K\lambda/\e^2\Big)\\
&+\sup_{x,y}\P_{x,y}\Big(\nu_\lambda^\e> K\lambda/\e^2\Big).
\ea
We follow the argument of Lemma \ref{lem:compactCOnt} above. First we note that by \eqref{e:tau00} is suffices to consider initial points 
$x\in\{a_k^\e\}$, $y\in\bR^n$. Then, we essentially repeat the calculations of Lemma~\ref{lem:compactCOnt} 
with fixed $\eta$ instead of large $R$ and small $\lambda$ instead of fixed $T$. 
For instance, the estimate \eqref{e:X1} with $\{\Delta_k^\e\}$ defined in \eqref{e:DeltaX} takes the form
\ba
\P_{x,y}\Big(\sup_{t\in [0,\lambda]}|X_t^\e-x|>\frac{\eta}{2}, \nu_\lambda^\e\leq K\lambda/ \e^2 \Big)
&\leq
\P_{x,y}\Big(\max_{k\leq  K\lambda\e^{-2}}\Big|\sum_{j=1}^k (X^\e_{\tau^\e_k}- X^\e_{\tau^\e_{k-1}})\Big|>\frac{\eta}{2} - 2\|d\|_\infty\e\Big)\\
&\leq \P_{x,y}\Big(\max_{k\leq  K\lambda\e^{-2}}\Big|\sum_{j=1}^k \Delta^\e_k \Big|>\frac{\eta}{2} - 2\|d\|_\infty\e -C_1K\lambda\Big)\\
&\leq \frac{ C_2 K\lambda }{(\eta/2 - 2\|d\|_\infty\e -C_1K\lambda)^2}
\to 0,\quad \e\to 0, \lambda\to 0.
\ea
The estimates for the process $Y^\e$ are obtained analogously.  
\end{proof}

\subsection{Convergence of finite dimensional distributions}

We prove the convergence of finite dimensional distributions by the martingale problem method.
Let $(X,Y)$ be the diffusion with the generator $\cL$, $X_0=x\in\bR$, $Y_0=y\in\bR^n$.

Let $f\in C^3_b(\bR^{n+1},\bR)$, and let  of the process $(X,Y)$, which is defined by equation \eqref{e:XYlim}.
We have to show that for any $l\geq 1$, $h_1,\dots h_l\in C_b(\bR^{n+1},\bR)$, and any $0\leq s_1<\cdots <s_l< s< t$
\ba
\label{e:mart}
% \E\Big[  f(X^\e_t, Y_t^\e)-f(X^\e_s, Y_s^\e)-\int_s^t \cL f(X^\e_u, Y_u^\e)\,\di u \Big|\rF_s^\e\Big]\to 0,\quad
\E_{x,y}\Big[ \Big(f(X^\e_t, Y_t^\e)-f(X^\e_s, Y_s^\e)-\int_s^t Af(X^\e_u, Y_u^\e)\,\di u\Big)\prod_{j=1}^l h_j(X^\e_{s_j}, Y_{s_j}^\e)\Big]\to 0,\quad
\e\to 0.
\ea
This convergence will essentially follow from the next lemma.

\begin{lem}
Let $\eta>0$ and $f\in C^3_b(\bR^{n+1},\bR)$. Then for $\e>0$ small enough
\ba
\E_{x,y} \int_0^{\tau^\e} \Big( \cL f(X^\e_u, Y_u^\e)-\eta\Big)\,\di u
\leq
\E_{x,y} f(X_{\tau^\e}^\e,Y_{\tau^\e}^\e)-f(x,y)\leq \E_{x,y} \int_0^{\tau^\e} \Big( \cL f(X^\e_u, Y_u^\e)+\eta\Big)\,\di u
\ea
uniformly over $x\in\{a_k^\e\}$ and $y\in\bR^n$.
\end{lem}
\begin{proof}
 Let for definiteness $x=0$ and $y\in\bR^n$. 
The Taylor formula yields
\ba
\E_{0,y} f(X^\e_{\tau^\e},Y^\e_{\tau^\e})-f(0,y)&=f_x(0,y)\E_{0,y} X^\e_\tau+ \sum_{i=1}^n f_{y^i}(0,y)\E_{0,y} (Y^{\e,i}_{\tau^\e}-y^i)\\
&+\frac{1}{2}f_{xx}(0,y)\E_{0,y} |X^\e_{\tau^\e}|^2+ \frac{1}{2}\sum_{i=1}^n f_{x y^i}(0,y)\E_{0,y} X^\e_{\tau^\e}(Y^{\e,i}_{\tau^\e}-y^i)\\
&+\frac12  \sum_{i,j=1}^n f_{y^iy^j}(0,y)\E_{0,y}  (Y^{\e,i}_{\tau^\e}-y^i) (Y^{\e,j}_{\tau^\e}-y^j)+Q^\e
\ea
where
\ba
|Q^\e|&\lesssim \E_{0,y} |X^\e_{\tau^\e}|^3 + \E_{0,y} |Y^\e_{\tau^\e}-y|^3 + \E_{0,y} |X^\e_{\tau^\e}|^2  |Y^\e_{\tau^\e}-y|
+\E_{0,y} |X^\e_{\tau^\e}|  |Y^\e_{\tau^\e}-y|^2\\
&\lesssim \e^3\leq \frac{\eta}{3} \E_{0,y} \tau^\e.
\ea
Taking into account the asymptotics from Sections \ref{s:X} and \ref{s:Y} and estimates from Lemma \ref{l:V}  we get
\ba
\E_{0,y} f(X^\e_{\tau^\e},Y^\e_{\tau^\e})-f(0,y)&=\Big(\cL  f(0,y) + \mathcal{O}(\e^{1-\delta})\Big)\cdot \E_{0,y}\tau^\e \\
&\leq \E_{0,y}\int_0^{\tau^\e}  \Big(\cL  f(0,y) + \frac{\eta}{3}\Big)\,\di u  \\
&=\E_{0,y}\int_0^{\tau^\e}  \Big(\cL  f(X_u^\e,Y^\e_u) + \cL  f(0,y) - \cL  f(X_u^\e,Y^\e_u)+ \frac{\eta}{3}\Big)\,\di u.
\ea
Since $(x,y)\mapsto \cL f(x,y)$ is Lipschitz continuous, we have
\ba
\sup_{u\leq \tau^\e}|\cL  f(0,y) - \cL  f(X_u^\e,Y^\e_u) |\leq C\Big(\e + \sup_{u\leq \tau^\e}|Y^\e_u-y|\Big)
\ea
and thus by Lemmas \ref{l:tau} and \ref{l:V}
\ba
\Big| \E_{0,y}\int_0^{\tau^\e}  \Big(\cL  f(0,y) - \cL  f(X_u^\e,Y^\e_u)\Big)\,\di u\Big|
&\lesssim \e\E_{0,y}\tau^\e + \Big(\E_{0,y}|\tau^\e|^2\cdot \E_{0,y} \sup_{u\leq \tau^\e}|Y^\e_u-y|^2\Big)^{1/2} \\
&\lesssim \e^3 \leq \frac{\eta}{3} \E_{0,y} \tau^\e.
\ea
The estimate from below follows analogously.
\end{proof}

Now we prove \eqref{e:mart}.
For $s\geq 0$, let $\tau^\e(s)$ be the first hitting time of a membrane after $s$, i.e.,
\ba
\tau^\e(s)=\tau_{\nu_s^\e}^\e=\min\{\tau_k^\e\colon \tau_k^\e\geq s\}.
\ea
We estimate the conditional expectation:
\ba
\E\Big[  f(X^\e_t, Y_t^\e)&-f(X^\e_s, Y_s^\e)-\int_s^t \cL f(X^\e_u, Y_u^\e)\,\di u \Big|\rF_{s}^\e\Big]\\
&=\sum_{k\geq 1}
\E\Big[ \bI_{\{s<\tau_{k}^\e\leq t\}}
\Big( f(X^\e_{\tau_{k+1}^\e}, Y_{\tau_{k+1}^\e}^\e)-f(X^\e_{\tau_{k}^\e}, Y_{\tau_{k}^\e}^\e)
-\int_{\tau_{k }^\e}^{\tau_{k+1}^\e} \cL f(X^\e_u, Y_u^\e)\,\di u\Big) \Big|\rF_{s}^\e\Big] \\
&+ \E \Big[  f(X^\e_{\tau^\e(s)}, Y_{\tau^\e(s)}^\e)-f(X^\e_{s}, Y_{s}^\e)-\int_{s}^{{\tau^\e(s)}} \cL f(X^\e_u, Y_u^\e)\,\di u \Big|\rF_{s}^\e\Big]\\
&- \E \Big[  f(X^\e_{\tau^\e(t)}, Y_{\tau^\e(t)}^\e)-f(X^\e_{t}, Y_{t}^\e)-\int_{t}^{{\tau^\e(t)}} \cL f(X^\e_u, Y_u^\e)\,\di u \Big|\rF_{s}^\e\Big]\\
&\leq \sum_{k\geq 1}
\E\Big[ \bI_{\{s<\tau_{k}^\e\leq t\}} \E\Big( f(X^\e_{\tau_{k+1}^\e}, Y_{\tau_{k+1}^\e}^\e)
-f(X^\e_{\tau_{k}^\e}, Y_{\tau_{k}^\e}^\e)-\int_{\tau_{k }^\e}^{\tau_{k+1}^\e} \cL f(X^\e_u, Y_u^\e)\,\di u \Big |\rF_{\tau_{k}^\e}\Big) \Big|\rF_{s}^\e\Big] \\
&+C\sup_{x,y}\E_{x,y} \Big[ |X^\e_{\tau^\e(0)}-x|+|Y_{\tau^\e(0)}^\e-y|+ \tau^\e(0)\Big]\\
&\leq    \delta  \sum_{k\geq 1}
\E\Big[ \bI_{\{s<\tau_{k}^\e\leq t\}} (\tau_{k+1}^\e-\tau_{k}^\e)\Big|\rF_{s}^\e\Big]
+  \mathcal{O}(\e)\\
&\leq 
\delta \E \Big[ \tau^\e(t)- \tau^\e(s)\Big|\rF_s\Big]+  \mathcal{O}(\e)\\
&\leq \delta (t-s)+\delta \sup_{x,y}\E_{x,y} (\tau^\e(t)- t) + \delta \sup_{x,y}\E_{x,y}( \tau^\e(s)-s) +    \mathcal{O}(\e) = \delta (t-s)+ \mathcal{O}(\e).
\ea
Here we used \eqref{e:tau00} implies the estimates
\ba
\sup_{x,y}\E_{x,y} (\tau^\e(t)- t)=\sup_{x,y}\E_{x,y} \tau^\e(0)\lesssim \e^2
\ea
and
\ba
\sup_{x,y}\E_{x,y} \Big[ |X^\e_{\tau^\e(0)}-x|+|Y_{\tau^\e(0)}^\e-y|\Big]\lesssim \e.
\ea
A similar estimate from below holds true, too, hence multiplying this conditional expectation by the functions $h_1,\dots,h_l$ and taking expectation
yields the limit \eqref{e:mart}.

\section*{Statements and Declarations}

\noindent
\textbf{Availability of data and material.} Data sharing not applicable to this article as no datasets
were generated or analyzed during the current study.

\medskip 

\noindent
\textbf{Conflict of interests.} The authors declare that they have no conflict of interest.

\medskip 
\noindent
\textbf{Authors' contributions.} All authors have contributed equally to the paper.

\providecommand{\bysame}{\leavevmode\hbox to3em{\hrulefill}\thinspace}
\providecommand{\MR}{\relax\ifhmode\unskip\space\fi MR }
\providecommand{\MRhref}[2]{%
  \href{http://www.ams.org/mathscinet-getitem?mr=#1}{#2}
}
\providecommand{\href}[2]{#2}

% 
% \bibliographystyle{amsplain}
% \bibliography{biblio-new}

\end{document}